\documentclass[reqno]{amsart}
\usepackage{amssymb}
\usepackage{amsmath}
\usepackage{amsfonts}

\newtheorem{theorem}{Theorem}[section]
\theoremstyle{plain}

\newtheorem{lemma}{Lemma}[section]

\newtheorem{remark}{Remark}[section]

\numberwithin{equation}{section}

\begin{document}
\title[Global attractors]{Global attractors for strongly damped wave
equations with displacement dependent damping and nonlinear source term of
critical exponent}
\author{A. Kh. Khanmamedov}
\address{Department of Mathematics, Faculty of Science, Hacettepe
University, Beytepe 06800, Ankara, Turkey.}
\email{azer@hacettepe.edu.tr}
\date{}
\thanks{}

\begin{abstract}
In this paper the long time behaviour of the solutions of the 3-D strongly
damped wave equation is studied. It is shown that the semigroup generated by
this equation possesses a global attractor in $H_{0}^{1}(\Omega )\times
L_{2}(\Omega )$ and then it is proved that this is also a global attractor
in $(H^{2}(\Omega )\cap H_{0}^{1}(\Omega ))\times H_{0}^{1}(\Omega )$.
\end{abstract}

\maketitle

\section{Introduction}

We consider the following initial-boundary value problem for the strongly
damped wave equation:%
\begin{equation}
w_{tt}-\Delta w_{t}+\sigma (w)w_{t}-\Delta w+f(w)=g(x)\text{\ \ \ \ \ \ in \
}(0,\infty )\times \Omega,  \tag{1.1}
\end{equation}%
\begin{equation}
w=0\text{ \ \ \ \ \ \ \ \ \ \ \ \ \ \ \ \ \ \ \ \ \ \ \ \ \ \ \ \ \ \ \ \ \
\ \ \ \ \ \ \ \ \ \ \ \ \ \ \ \ \ on \ }(0,\infty )\times \partial \Omega,
\tag{1.2}
\end{equation}%
\begin{equation}
w(0,\cdot )=w_{0}\text{ },\text{ \ \ \ \ \ \ \ \ \ \ \ \ \ \ \ \ \ \ \ \ \ \
}w_{t}(0,\cdot )=w_{1}\text{\ \ \ \ \ \ \ \ \ \ \ \ \ \ \ in \ }\Omega,
\tag{1.3}
\end{equation}%
where $\Omega \subset R^{3}$ is a bounded domain with sufficiently smooth
boundary and $g\in L_{2}(\Omega )$.

As shown in \cite{6} and \cite{13}, equation (1.1) is related to the
following reaction-diffusion equation with memory:%
\begin{equation}
w_{t}(t,x)=\overset{t}{\underset{-\infty }{\int }}K(t,s)\Delta
w(s,x)ds-f(w(t,x))+g(x).  \tag{1.4}
\end{equation}%
Namely, if $K(t,s)=\frac{1-\alpha }{\lambda }e^{-\frac{t-s}{\lambda }%
}+2\alpha \delta (t-s)$ then (1.4) can be transformed into%
\begin{equation*}
\lambda w_{tt}-\alpha \lambda \Delta w_{t}+(1+\lambda f^{\prime
}(w))w_{t}-\Delta w+f(w)=g,
\end{equation*}%
where $\lambda >0,$ $\alpha \in \lbrack 0,1)$ and $\delta $ is a Dirac delta
function. This equation is interesting from a physical viewpoint as a model
describing the flow of viscoelastic fluids (see \cite{6} and \cite{13} for
details).

When $\sigma (\cdot )\equiv 0$ the equation (1.1) becomes
\begin{equation}
w_{tt}-\Delta w_{t}-\Delta w+f(w)=g.\text{\ }  \tag{1.5}
\end{equation}%
The long time behaviour (in terms of attractors) of solutions in this case
has been studied by many authors (see \cite{2}, \cite{5}, \cite{7}, \cite{14}%
, \cite{15}, \cite{19}, \cite{22} and references therein). In \cite{14} the
existence of a global attractor for (1.5) with critical source term (i.e. in
the case when the growth of $f$ is of order 5) was proved. However, the
regularity of the global attractor in that article was established only in
the subcritical case. For the critical case, the regularity of the global
attractor of (1.5) was proved in \cite{15}, under the assumptions
\begin{equation}
f\in C^{1}(R)\text{, }\left\vert f^{\prime }(s)\right\vert \leq
c(1+\left\vert s\right\vert ^{4})\text{, }\forall s\in R\text{ and }\underset%
{\left\vert s\right\vert \rightarrow \infty }{\liminf }f^{\prime
}(s)>-\lambda _{{\small 1}}  \tag{1.6}
\end{equation}%
or
\begin{equation}
f\in C^{2}(R)\text{, }\left\vert f^{\prime \prime }(s)\right\vert \leq
c(1+\left\vert s\right\vert ^{3})\text{, }\forall s\in R\text{ and }\underset%
{\left\vert s\right\vert \rightarrow \infty }{\liminf }\frac{f(s)}{s}%
>-\lambda _{{\small 1}},  \tag{1.7}
\end{equation}%
where $\lambda _{{\small 1}}$ is a first eigenvalue of $-\Delta $ with zero
Dirichlet data. In that article the authors obtained a regular estimate for $%
w_{tt}$ (when $w(t,x)$ is a weak solution of (1.5)) and then proved the
asymptotic regularity of the solution of the non-autonomous equation
\begin{equation*}
-\Delta w_{t}-\Delta w+f(w)=g-w_{tt}.
\end{equation*}%
In \cite{5} and \cite{19}, the regularity of the global attractor of (1.5)
was proved under the following weaker condition on the source term:%
\begin{equation*}
f\in C(R)\text{, }|f(u)-f(v)|\leq c(1+|u|^{4}+|v|^{4})|u-v|\text{, }\forall
u,v\in R\text{ and }\underset{\left\vert s\right\vert \rightarrow \infty }{%
\liminf }\frac{f(s)}{s}>-\lambda _{{\small 1}}\text{.}
\end{equation*}%
In \cite{8}, the authors investigated the weak attractor for the
quasi-linear strongly damped equation%
\begin{equation*}
w_{tt}-\Delta w_{t}-\Delta w+f(w)=\nabla \cdot \varphi ^{\prime }(\nabla w)+g
\end{equation*}%
under the following conditions on the nonlinear functions $f$ and $\varphi $:%
\begin{equation*}
f\in C^{1}(R)\text{, \ }-C+a_{1}\left\vert s\right\vert ^{q}\leq f^{\prime
}(s)\leq C\left\vert s\right\vert ^{q}\text{, }\forall s\in R\text{,}
\end{equation*}%
\begin{equation*}
\varphi \in C^{2}(R^{3},R)\text{, \ }a_{2}\left\vert \eta \right\vert
^{p-1}\left\vert \xi \right\vert ^{2}\leq \sum_{i,j=1}^{3}\frac{\partial
^{2}\varphi (\eta )}{\partial \eta _{i}\partial \eta _{j}}\xi _{i}\xi
_{j}\leq a_{3}(1+\left\vert \eta \right\vert ^{p-1})\left\vert \xi
\right\vert ^{2}\text{, \ }\forall \xi ,\eta \in R^{3}\text{,}
\end{equation*}%
for some $a_{i}>0$, ($i=1,2,3$), $C>0$, $q>0$ and $p\in \lbrack 1,5)$. When $%
\frac{\partial ^{2}\varphi }{\partial \eta _{i}\partial \eta _{j}}=0$, ($%
i,j=1,2,3$), the strong attractor has also been studied. Recently, in \cite%
{3}, the authors have studied the global attractor for the strongly damped
abstract equation
\begin{equation*}
w_{tt}+D(w,w_{t})+Aw+F(w)=0.
\end{equation*}
However, the approaches of the articles mentioned above, in general, do not
seem to be applicable to (1.1). The difficulty is caused by the term $\sigma
(w)w_{t}$, when the function $\sigma (\cdot )$ is not differentiable and the
growth condition imposed on $\sigma (\cdot )$ is critical. In this paper we
prove the existence of the global attractors for (1.1)-(1.3) in $%
H_{0}^{1}(\Omega )\times L_{2}(\Omega )$ and $(H^{2}(\Omega )\cap
H_{0}^{1}(\Omega ))\times H_{0}^{1}(\Omega )$. Then using the embedding $H^{%
\frac{3}{2}+\varepsilon }(\Omega )\subset C(\overline{\Omega })$ we show
that these attractors coincide.

\section{Well-posedness and the statement of the main result}

We start with the conditions on nonlinear terms $f$ and $\sigma $.%
\begin{equation}
\bullet \text{ }f\in C(R),\text{ }\left\vert f(s)-f(t)\right\vert \leq
c(1+\left\vert s\right\vert ^{4}+\left\vert t\right\vert ^{4})\left\vert
s-t\right\vert \text{, }\forall s,t\in R\text{, }  \tag{2.1}
\end{equation}%
\begin{equation}
\bullet \ \underset{\left\vert s\right\vert \rightarrow \infty }{\liminf }%
\frac{f(s)}{s}>-\lambda _{{\small 1}},\text{ where }\lambda _{{\small 1}}=%
\underset{\varphi \in H_{0}^{1}(\Omega ),\varphi \neq 0}{\inf }\frac{%
\left\Vert \nabla \varphi \right\Vert _{L_{{\small 2}}(\Omega )}^{2}}{%
\left\Vert \varphi \right\Vert _{L_{{\small 2}}(\Omega )}^{2}}\text{ ,\ \ \
\ \ \ \ \ }  \tag{2.2}
\end{equation}

\begin{equation}
\bullet \text{ }\sigma \in C(R)\text{, \ }\sigma (s)\geq 0\text{, \ }%
\left\vert \sigma (s)\right\vert \leq c(1+\left\vert s\right\vert ^{4})\text{%
, \ }\forall s\in R\text{. \ \ \ \ \ \ \ \ \ \ \ \ }  \tag{2.3}
\end{equation}%
By the standard Galerkin's method it is easy to prove the following
existence theorem:

\begin{theorem}
Let conditions (2.1)-(2.3) hold. Then for every $T>0$ and every $%
(w_{0},w_{1})\in \mathcal{H}:=H_{0}^{1}(\Omega )\times L_{2}(\Omega )$, the
problem (1.1)-(1.3) admits a weak solution
\begin{equation*}
w\in C([0,T];H_{0}^{1}(\Omega )),\text{ }w_{t}\in C([0,T];L_{2}(\Omega
))\cap L_{2}(0,T;H_{0}^{1}(\Omega )),
\end{equation*}%
which satisfies the following energy equality%
\begin{equation*}
E(w(t))+\overset{t}{\underset{s}{\int }}\left\Vert \nabla w_{t}(\tau
)\right\Vert _{L_{2}(\Omega )}^{{\small 2}}d\tau +\overset{t}{\underset{s}{%
\int }}\left\langle \sigma (w(\tau ))w_{t}(\tau ),w_{t}(\tau )\right\rangle
d\tau +\left\langle F(w(t)),1\right\rangle -
\end{equation*}%
\begin{equation}
-\left\langle g,w(t)\right\rangle =E(w(s))+\left\langle
F(w(s)),1\right\rangle -\left\langle g,w(s)\right\rangle ,\text{ \ \ \ \ \ \
\ }0\leq s\leq t\leq T,  \tag{2.4}
\end{equation}%
where $E(w(t))=\frac{{\small 1}}{{\small 2}}(\left\Vert \nabla
w(t)\right\Vert _{L_{2}(\Omega )}^{{\small 2}}+\left\Vert
w_{t}(t)\right\Vert _{L_{2}(\Omega )}^{{\small 2}}),$ $\left\langle
u,v\right\rangle =\underset{\Omega }{\int }$ $u(x)v(x)dx$ and $F(w)=\overset{%
w}{\underset{0}{\int }}f(u)du.$
\end{theorem}

Now using the method of \cite[Proposition 2.2]{16} let us prove the
following uniqueness theorem:

\begin{theorem}
Let conditions (2.1)-(2.3) hold. If $w(t,\cdot)$ and $\widehat{w}(t,\cdot)$
are the weak solutions of (1.1)-(1.3), determined by Theorem 2.1, with
initial data $(w_{0},w_{1})$ and $(\widehat{w}_{0},\widehat{w}_{1})$
respectively, then%
\begin{equation*}
\left\Vert w(T)-\widehat{w}(T)\right\Vert _{H^{1}(\Omega )}^{2}+\left\Vert
w_{t}(T)-\widehat{w}_{t}(T)\right\Vert _{H^{-1}(\Omega )}^{2}\leq
\end{equation*}%
\begin{equation*}
\leq c(T,R)\left( \left\Vert w_{0}-\widehat{w}_{0}\right\Vert _{H^{1}(\Omega
)}+\left\Vert w_{1}-\widehat{w}_{1}\right\Vert _{H^{-1}(\Omega )}\right)
\end{equation*}%
where $c:R_{+}\times R_{+}\rightarrow R_{+}$ is a nondecreasing function
with respect to each variable and $R=\max \left\{ \left\Vert
(w_{0},w_{1})\right\Vert _{\mathcal{H}},\left\Vert (\widehat{w}_{0},\widehat{%
w}_{1})\right\Vert _{\mathcal{H}}\right\} $.
\end{theorem}

\begin{proof}
By (2.1)-(2.4), it follows that
\begin{equation*}
\left\Vert (w(t),w_{t}(t))\right\Vert _{\mathcal{H}}+\left\Vert (\widehat{w}%
(t),\widehat{w}_{t}(t))\right\Vert _{\mathcal{H}}\leq c_{1}(R),\text{ \ \ }%
\forall t\geq 0.
\end{equation*}%
Denote $u(t,\cdot)=w(t,\cdot)-$ $\widehat{w}(t,\cdot)$ and $\widehat{u}%
(t,\cdot)=$ $\overset{t}{\underset{0}{\int }}u(\tau ,\cdot)d\tau $.
Integrating (1.1) for $w(t,\cdot)$ and $\widehat{w}(t,\cdot)$ on $[0,t]$ and
taking the difference, we have%
\begin{equation*}
u_{t}-\Delta u+\Sigma (w)-\Sigma (\widehat{w})-\Delta \widehat{u}+\overset{t}%
{\underset{0}{\int }}\left( f(w(\tau ,))-f(\widehat{w}(\tau ,))\right) d\tau
=
\end{equation*}%
\begin{equation}
=\Sigma (w_{0})-\Sigma (\widehat{w}_{0})-\Delta (w_{0}-\widehat{w}%
_{0})+w_{1}-\widehat{w}_{1},\text{ \ \ \ }\forall t\geq 0,  \tag{2.5}
\end{equation}%
where $\Sigma (w)=$ $\overset{w}{\underset{0}{\int }}\sigma (s)ds$. Testing
(2.5) by $u$ and taking into account (2.1), (2.3), (2.4) and monotonicity of
$\Sigma (\cdot )$, we find%
\begin{equation*}
\frac{d}{dt}E(\widehat{u}(t))+\frac{1}{2}\left\Vert \nabla u(t)\right\Vert
_{L_{2}(\Omega )}^{2}\leq
\end{equation*}%
\begin{equation*}
\leq c_{2}(R)\left( \left\Vert \nabla (w_{0}-\widehat{w}_{0})\right\Vert
_{L_{2}(\Omega )}^{2}+\left\Vert w_{1}-\widehat{w}_{1}\right\Vert
_{H^{-1}(\Omega )}^{2}\right) +
\end{equation*}%
\begin{equation}
+c_{2}(R)t\overset{t}{\underset{0}{\int }}\left\Vert \nabla u(\tau
)\right\Vert _{L_{2}(\Omega )}^{2}d\tau ,\text{ \ }\forall t\geq 0  \tag{2.6}
\end{equation}%
and consequently
\begin{equation*}
\frac{d}{dt}\widehat{E}(\widehat{u}(t))\leq c_{2}(R)\left( \left\Vert w_{0}-%
\widehat{w}_{0}\right\Vert _{H^{1}(\Omega )}^{2}+\left\Vert w_{1}-\widehat{w}%
_{1}\right\Vert _{H^{-1}(\Omega )}^{2}\right) +2c_{2}(R)t\widehat{E}(%
\widehat{u}(t)),
\end{equation*}%
where $\widehat{E}(\widehat{u}(t))=E(\widehat{u}(t))+\frac{1}{2}\overset{t}{%
\underset{0}{\int }}\left\Vert \nabla u(\tau )\right\Vert _{L_{2}(\Omega
)}^{2}d\tau $. Applying Gronwall's lemma to the last inequality, we get%
\begin{equation}
\widehat{E}(\widehat{u}(t))\leq c_{3}(R)e^{c_{2}(R)t^{2}}\left( \left\Vert
w_{0}-\widehat{w}_{0}\right\Vert _{H^{1}(\Omega )}^{2}+\left\Vert w_{1}-%
\widehat{w}_{1}\right\Vert _{H^{-1}(\Omega )}^{2}\right)  \tag{2.7}
\end{equation}%
By (2.1), (2.3), (2.4) and (2.7), it follows that%
\begin{equation*}
\left\vert \frac{d}{dt}E(\widehat{u}(t))\right\vert \leq \left\vert
\left\langle u_{t}(t),u(t)\right\rangle \right\vert +\left\vert \left\langle
\nabla \widehat{u}(t),\nabla u(t)\right\rangle \right\vert \leq
\end{equation*}%
\begin{equation*}
\leq c_{4}(R)\left( \left\Vert u(t)\right\Vert _{L_{2}(\Omega )}+\left\Vert
\nabla \widehat{u}(t)\right\Vert _{L_{2}(\Omega )}\right) \leq
\end{equation*}%
\begin{equation*}
\leq c_{5}(R)e^{\frac{c_{2}(R)t^{2}}{2}}\left( \left\Vert w_{0}-\widehat{w}%
_{0}\right\Vert _{H^{1}(\Omega )}+\left\Vert w_{1}-\widehat{w}%
_{1}\right\Vert _{H^{-1}(\Omega )}\right) ,\text{ \ \ }\forall t\geq 0.
\end{equation*}%
Taking into account (2.7) and the last inequality in (2.6), we obtain%
\begin{equation*}
\left\Vert \nabla u(t)\right\Vert _{L_{2}(\Omega )}^{2}\leq
c_{6}(R)(1+t)e^{c_{2}(R)t^{2}}\left( \left\Vert w_{0}-\widehat{w}%
_{0}\right\Vert _{H^{1}(\Omega )}\right. +
\end{equation*}%
\begin{equation*}
+\left. \left\Vert w_{1}-\widehat{w}_{1}\right\Vert _{H^{-1}(\Omega
)}\right) ,\text{ \ \ }\forall t\geq 0.
\end{equation*}%
Now, from (2.5), we have%
\begin{equation*}
\left\Vert u_{t}(t)\right\Vert _{H^{-1}(\Omega )}\leq \left\Vert \nabla
u(t)\right\Vert _{L_{2}(\Omega )}+\left\Vert \nabla \widehat{u}%
(t)\right\Vert _{L_{2}(\Omega )}+\left\Vert \Sigma (w(t))-\Sigma (\widehat{w}%
(t))\right\Vert _{H^{-1}(\Omega )}+
\end{equation*}%
\begin{equation*}
+\overset{t}{\underset{0}{\int }}\left\Vert f(w(\tau ,))-f(\widehat{w}(\tau
,))\right\Vert _{H^{-1}(\Omega )}d\tau +\left\Vert \Sigma (w_{0})-\Sigma (%
\widehat{w}_{0})\right\Vert _{H^{-1}(\Omega )}+
\end{equation*}%
\begin{equation*}
+\left\Vert \nabla (w_{0}-\widehat{w}_{0})\right\Vert _{L_{2}(\Omega
)}+\left\Vert w_{1}-\widehat{w}_{1}\right\Vert _{H^{-1}(\Omega )},\text{ \ \
\ }
\end{equation*}%
which due to the above inequalities gives%
\begin{equation*}
\left\Vert u_{t}(t)\right\Vert _{H^{-1}(\Omega )}^{2}\leq
c_{7}(R)(1+t)e^{c_{2}(R)t^{2}}\left( \left\Vert w_{0}-\widehat{w}%
_{0}\right\Vert _{H^{1}(\Omega )}\right. +
\end{equation*}%
\begin{equation*}
+\left. \left\Vert w_{1}-\widehat{w}_{1}\right\Vert _{H^{-1}(\Omega
)}\right) ,\text{ \ \ }\forall t\geq 0.
\end{equation*}
\end{proof}

Thus by Theorem 2.1 and Theorem 2.2, it follows that by the formula $%
S(t)(w_{0},w_{1}) \newline
=(w(t),w_{t}(t))$, problem (1.1)-(1.3) generates a weakly continuous (in the
sense, if $\varphi _{n}\rightarrow \varphi $ strongly then $S(t)\varphi
_{n}\rightarrow S(t)\varphi $ weakly) semigroup $\left\{ S(t)\right\}
_{t\geq 0}$ in $\mathcal{H}$, where $w(t,\cdot)$ is a weak solution of
(1.1)-(1.3), determined by Theorem 2.1, with initial data $(w_{0},w_{1})$.
To show the strong continuity of $\left\{ S(t)\right\} _{t\geq 0}$ we
firstly prove the following lemma:

\begin{lemma}
Let $\varphi \in C(R)$ and $\left\vert \varphi (x)\right\vert \leq
c(1+\left\vert x\right\vert ^{r})$ for every $x\in R$ and some $r\geq 1$. If
$v_{n}\rightarrow v$ strongly in $L_{q}(\Omega )$ for $q\geq r$, then $%
\varphi (v_{n})\rightarrow \varphi (v)$ strongly in $L_{\frac{q}{r}}(\Omega
) $.
\end{lemma}

\begin{proof}
By the assumption of the lemma, there exists a subsequence $\left\{
v_{n_{k}}\right\} $ such that $v_{n_{k}}\rightarrow v$ a.e. in $\Omega $.
Then by Egorov's theorem, for any $\varepsilon >0$ there exists a subset $%
A_{\varepsilon }\subset \Omega $ such that $mes(A_{\varepsilon
})<\varepsilon $ and $v_{n_{k}}\rightarrow v$ uniformly in $\Omega
\backslash A_{\varepsilon }$. Hence for large enough $k$%
\begin{equation*}
\left\vert v_{n_{k}}(x)\right\vert \leq 1+\left\vert v(x)\right\vert \text{
\ in }\Omega \backslash A_{\varepsilon }
\end{equation*}%
and consequently
\begin{equation*}
\left\vert \varphi (v_{n_{k}}(x))\right\vert \leq c_{1}(1+\left\vert
v(x)\right\vert ^{r})\text{ \ in }\Omega \backslash A_{\varepsilon }\text{.}
\end{equation*}%
Applying Lebesgue's theorem we get
\begin{equation}
\underset{k\rightarrow \infty }{\lim }\left\Vert \varphi (v_{n_{k}})-\varphi
(v)\right\Vert _{L_{\frac{q}{r}}(\Omega \backslash A_{\varepsilon })}=0\text{%
.}  \tag{2.8}
\end{equation}%
On the other hand since we have
\begin{equation*}
\underset{k\rightarrow \infty }{\lim }\left\Vert v_{n_{k}}\right\Vert
_{L_{q}(A_{\varepsilon })}=\left\Vert v\right\Vert _{L_{q}(A_{\varepsilon })}%
\text{,}
\end{equation*}%
the inequality
\begin{equation*}
\underset{k\rightarrow \infty }{\lim \sup }\left\Vert \varphi
(v_{n_{k}})\right\Vert _{L_{\frac{q}{r}}(A_{\varepsilon })}^{\frac{q}{r}%
}<c_{3}(\varepsilon +\left\Vert v\right\Vert _{L_{q}(A_{\varepsilon })}^{q})
\end{equation*}%
is satisfied. The last inequality together with (2.8) implies that
\begin{equation*}
\underset{k\rightarrow \infty }{\lim \sup }\left\Vert \varphi
(v_{n_{k}})-\varphi (v)\right\Vert _{L_{\frac{q}{r}}(\Omega )}^{\frac{q}{r}%
}\leq c_{4}\underset{\varepsilon \rightarrow 0}{\lim }(\varepsilon
+\left\Vert v\right\Vert _{L_{q}(A_{\varepsilon })}^{q})=0\text{.}
\end{equation*}
\end{proof}

\begin{theorem}
Under conditions (2.1)-(2.3) the semigroup $\left\{ S(t)\right\} _{t\geq 0}$
is strongly continuous in $\mathcal{H}$.
\end{theorem}

\begin{proof}
Let $(w_{0n},w_{1n})\rightarrow (w_{0},w_{1})$ strongly in $\mathcal{H}$.
Denoting $(w_{n}(t),w_{tn}(t))=$\newline
$S(t)(w_{0n},w_{1n})$, $(w(t),w_{t}(t))=S(t)(w_{0},w_{1})$ and $%
u_{n}(t)=w_{n}(t)-w(t)$, by (1.1) we have%
\begin{equation*}
u_{ntt}-\Delta u_{nt}+\sigma (w_{n})w_{nt}-\sigma (w)w_{t}-\Delta
u_{n}+f(w_{n}(\tau ))-f(w(t))=0\text{.}
\end{equation*}%
Since, by Theorem 2.1, every term of the above equation belongs to $%
L_{2}(0,T;H^{-1}(\Omega ))$, testing it by $u_{nt}$, we obtain%
\begin{equation*}
E(u_{n}(t))\leq E(u_{n}(0))+c\left\Vert \sigma (w_{n})-\sigma (w)\right\Vert
_{C([0,T];L_{\frac{3}{2}}(\Omega ))}^{2}+c\overset{t}{\underset{0}{\int }}%
E(u_{n}(s))ds\text{, }\forall t\in \lbrack 0,T]\text{.}
\end{equation*}%
Applying Gronwall's lemma we have
\begin{equation}
E(u_{n}(T))\leq \left( E(u_{n}(0))+c\left\Vert \sigma (w_{n})-\sigma
(w)\right\Vert _{C([0,T];L_{\frac{3}{2}}(\Omega ))}^{2}\right) e^{cT}\text{,
\ }\forall T\geq 0\text{.}  \tag{2.9}
\end{equation}
By Theorem 2.2, it follows that%
\begin{equation*}
\underset{n\rightarrow \infty }{\lim }\left\Vert w_{n}-w\right\Vert
_{C([0,T];L_{6}(\Omega ))}=0\text{.}
\end{equation*}%
Now applying Lemma 2.1 it is easy to see that
\begin{equation*}
\underset{n\rightarrow \infty }{\lim }\left\Vert \sigma (w_{n})-\sigma
(w)\right\Vert _{C([0,T];L_{\frac{3}{2}}(\Omega ))}=0\text{,}
\end{equation*}%
which together with (2.9) yields that $S(T)(w_{0n},w_{1n})\rightarrow
S(T)(w_{0},w_{1})$ strongly in $\mathcal{H}$, for every $T\geq 0$.
\end{proof}

Now let us recall the definition of a global attractor.\newline
\textbf{Definition (\cite{17}).} Let $\{V(t)\}_{{\small t\geq 0}}$ be a
semigroup on a metric space $(X,$ $d).$ A compact set $\mathcal{A}\subset X$
is called a global attractor for the semigroup $\left\{ V(t)\right\} _{t\geq
0}$ iff

$\bullet $ $\mathcal{A}$ is invariant, i.e. $V(t)\mathcal{A}=\mathcal{A},$ $%
\forall t\geq 0;$

$\bullet $ $\underset{t\rightarrow \infty }{\lim }$\ $\underset{v\in B}{\sup
}$ $\underset{u\in \mathcal{A}}{\inf }d(V(t)v,u)=0$ \ for each bounded set $%
B\subset X.$\newline

Our main result is as follows:

\begin{theorem}
Under the conditions (2.1)-(2.3), the semigroup $\{S(t)\}_{t\geq 0}$
generated by the problem (1.1)-(1.3) possesses a global attractor $\mathcal{A%
}$ in $\mathcal{H}$, which is also a global attractor in $\mathcal{H}%
_{1}:=(H^{2}(\Omega )\cap H_{0}^{1}(\Omega ))\times H_{0}^{1}(\Omega )$.%
\newline
\end{theorem}

\begin{remark}
We note that if the condition (2.3) is replaced by
\begin{equation*}
\sigma \in C(R)\text{, \ }\sigma (s)\geq 0\text{, }\left\vert \sigma
(s)\right\vert \leq c(1+\left\vert s\right\vert ^{p})\text{, }0\leq p<4\text{%
, }\forall s\in R\text{,}
\end{equation*}%
then using the methods of \cite{5} , \cite{19} and \cite{21} one can prove
Theorem 2.4. If we assume
\begin{equation*}
\sigma \in C^{1}(R)\text{, \ }\sigma (s)\geq 0\text{, }\left\vert \sigma
^{\prime }(s)\right\vert \leq c(1+\left\vert s\right\vert )\text{, }\forall
s\in R\text{,}
\end{equation*}%
instead of (2.3), then the method of \cite{15} can be applied to
(1.1)-(1.3). In this case, as in \cite{20}, one can show that a global
attractor $\mathcal{A}$ attracts every bounded subset of $\mathcal{H}$ in
the topology of $H_{0}^{1}(\Omega )\times H_{0}^{1}(\Omega )$.
\end{remark}

\begin{remark}
We also note that problem (1.1)-(1.3), in 3-D case, without the strong
damping $-\Delta w_{t}$ was considered in \cite{11} and \cite{16}. In this
case, when $\sigma (\cdot )$ is not globally bounded, the existence of a
global attractor in the strong topology of $\mathcal{H}$ and the regularity
of the weak attractor remain open (see \cite{11} and \cite{16} for details).
\end{remark}

\section{Existence of the global attractor in $\mathcal{H}$}

We start with the following asymptotic compactness lemma:

\begin{lemma}
Let conditions (2.1)-(2.3) hold and $B$ be a bounded subset of $\mathcal{H}$%
. Then every sequence of the form $\left\{ S(t_{n})\varphi _{n}\right\}
_{n=1}^{\infty }$, $\left\{ \varphi _{n}\right\} _{n=1}^{\infty }\subset B$,
$t_{n}\rightarrow \infty $, has a convergent subsequence in $\mathcal{H}$.
\end{lemma}

\begin{proof}
By (2.4), we have
\begin{equation}
\left\{
\begin{array}{c}
\underset{t\geq 0}{\sup }\underset{\varphi \in B}{\sup }\left\Vert
S(t)\varphi \right\Vert _{\mathcal{H}}<\infty, \\
\underset{\varphi \in B}{\sup }\underset{0}{\overset{\infty }{\int }}%
\left\Vert PS(t)\varphi \right\Vert _{H_{0}^{1}(\Omega )}^{2}dt<\infty,%
\end{array}%
\right.  \tag{3.1}
\end{equation}%
where $P:\mathcal{H\rightarrow }L_{2}(\Omega )$ is a projection map, i.e. $%
P\varphi =\varphi _{2}$, for every $\varphi =(\varphi _{1},\varphi _{2})\in
\mathcal{H}$. So for any $T_{0}\geq 1$ there exists a subsequence $\left\{
n_{k}\right\} _{k=1}^{\infty }$ such that $t_{n_{k}}\geq T_{0}$ and
\begin{equation}
\left\{
\begin{array}{c}
w_{k}\rightarrow w\text{ \ weakly star in }L_{\infty }(0,\infty
;H_{0}^{1}(\Omega )), \\
w_{kt}\rightarrow w_{t}\text{ \ \ \ \ \ weakly in }L_{2}(0,\infty
;H_{0}^{1}(\Omega )),%
\end{array}%
\right. \text{ }  \tag{3.2}
\end{equation}%
for some $w\in L_{\infty }(0,\infty ;H_{0}^{1}(\Omega ))\cap W^{1,\infty
}(0,\infty ;L_{2}(\Omega ))\cap W_{loc}^{1,2}(0,\infty ;H_{0}^{1}(\Omega ))$%
, where $(w_{k}(t),w_{kt}(t))=S(t+t_{n_{k}}-T_{0})\varphi _{n_{k}}$. Now
multiplying the equality
\begin{equation*}
(w_{k}-w_{m})_{tt}-\Delta (w_{kt}-w_{mt})+\sigma (w_{k})w_{kt}-\sigma
(w_{m})w_{mt}-\Delta (w_{k}-w_{m})+\text{\ }
\end{equation*}%
\begin{equation*}
+f(w_{k})-f(w_{m})=0
\end{equation*}%
by $(w_{kt}-w_{mt}+\frac{\lambda _{1}}{2}(w_{k}-w_{m}))$ and integrating
over $(s,T)\times \Omega $, we obtain%
\begin{equation*}
\frac{1}{2}E(w_{k}(T)-w_{m}(T))+\lambda _{1}\underset{s}{\overset{T}{\int }}%
E(w_{k}(t)-w_{m}(t))dt+
\end{equation*}%
\begin{equation*}
+\underset{s}{\overset{T}{\int }}\left\langle \sigma
(w_{k}(t))w_{kt}(t)-\sigma
(w_{m}(t))w_{mt}(t),w_{kt}(t)-w_{mt}(t)\right\rangle dt+
\end{equation*}%
\begin{equation*}
+\frac{\lambda _{1}}{2}\left\langle \widehat{\Sigma }(w_{k}(T))+\widehat{%
\Sigma }(w_{m}(T)),1\right\rangle -\frac{\lambda _{1}}{2}\underset{s}{%
\overset{T}{\int }}\left\langle \sigma
(w_{k}(t))w_{kt}(t),w_{m}(t)\right\rangle dt
\end{equation*}%
\begin{equation*}
-\frac{\lambda _{1}}{2}\underset{s}{\overset{T}{\int }}\left\langle \sigma
(w_{m}(t))w_{mt}(t),w_{k}(t)\right\rangle dt+\left\langle
F(w_{k}(T))+F(w_{m}(T)),1\right\rangle -
\end{equation*}%
\begin{equation*}
-\underset{s}{\overset{T}{\int }}\left\langle
f(w_{k}(t)),w_{mt}(t)\right\rangle dt-\underset{s}{\overset{T}{\int }}%
\left\langle f(w_{m}(t)),w_{kt}(t)\right\rangle dt+
\end{equation*}%
\begin{equation*}
+\frac{\lambda _{1}}{2}\underset{s}{\overset{T}{\int }}\left\langle
f(w_{k}(t))-f(w_{m}(t)),w_{k}(t)-w_{m}(t)\right\rangle dt\leq
\end{equation*}%
\begin{equation*}
\leq (\frac{3}{2}+\lambda _{1})E(w_{k}(s)-w_{m}(s))+\frac{\lambda _{1}}{2}%
\left\langle \widehat{\Sigma }(w_{k}(s))+\widehat{\Sigma }%
(w_{m}(s)),1\right\rangle +
\end{equation*}%
\begin{equation*}
+\left\langle F(w_{k}(s))+F(w_{m}(s)),1\right\rangle ,\ \ \ 0\leq s\leq T,
\end{equation*}%
where $\widehat{\Sigma }(w)=\underset{0}{\overset{w}{\int }}s\sigma (s)ds$.
Integrating the last inequality with respect to $s$ from $0$ to $T$ \ we find%
\begin{equation*}
\frac{T}{2}E(w_{k}(T)-w_{m}(T))+\lambda _{1}\underset{0}{\overset{T}{\int }}%
sE(w_{k}(s)-w_{m}(s))ds+
\end{equation*}%
\begin{equation*}
+\underset{0}{\overset{T}{\int }s}\left\langle \sigma
(w_{k}(s))w_{kt}(s)-\sigma
(w_{m}(s))w_{mt}(s),w_{kt}(s)-w_{mt}(s)\right\rangle ds+
\end{equation*}%
\begin{equation*}
+\frac{\lambda _{1}}{2}T\left\langle \widehat{\Sigma }(w_{k}(T))+\widehat{%
\Sigma }(w_{m}(T)),1\right\rangle -\frac{\lambda _{1}}{2}\underset{0}{%
\overset{T}{\int }}s\left\langle \sigma
(w_{k}(s))w_{kt}(s),w_{m}(s)\right\rangle ds
\end{equation*}%
\begin{equation*}
-\frac{\lambda _{1}}{2}\underset{0}{\overset{T}{\int }}s\left\langle \sigma
(w_{m}(s))w_{mt}(s),w_{k}(s)\right\rangle ds+T\left\langle
F(w_{k}(T))+F(w_{m}(T)),1\right\rangle -
\end{equation*}%
\begin{equation*}
-\underset{0}{\overset{T}{\int }}s\left\langle
f(w_{k}(s)),w_{mt}(s)\right\rangle ds-\underset{0}{\overset{T}{\int }}%
s\left\langle f(w_{m}(s)),w_{kt}(s)\right\rangle ds+
\end{equation*}%
\begin{equation*}
+\frac{\lambda _{1}}{2}\underset{0}{\overset{T}{\int }}s\left\langle
f(w_{k}(s))-f(w_{m}(s)),w_{k}(s)-w_{m}(s)\right\rangle dt\leq
\end{equation*}%
\begin{equation*}
\leq (\frac{3}{2}+\lambda _{1})\underset{0}{\overset{T}{\int }}%
E(w_{k}(s)-w_{m}(s))ds+\underset{0}{\overset{T}{\int }}\left\langle
F(w_{k}(s))+\frac{\lambda _{1}}{2}\widehat{\Sigma }(w_{k}(s)),1\right\rangle
ds+
\end{equation*}%
\begin{equation}
+\underset{0}{\overset{T}{\int }}\left\langle F(w_{m}(s))+\frac{\lambda _{1}%
}{2}\widehat{\Sigma }(w_{m}(s)),1\right\rangle ds,\text{ \ \ \ }\forall
T\geq 0.  \tag{3.3}
\end{equation}%
By (3.1)$_{1}$, it follows that
\begin{equation*}
(\frac{3}{2}+\lambda _{1})\underset{0}{\overset{T}{\int }}%
E(w_{k}(s)-w_{m}(s))ds\leq c_{1}+
\end{equation*}%
\begin{equation}
+\frac{\lambda _{1}}{2}\underset{\frac{3+2\lambda _{1}}{\lambda _{1}}}{%
\overset{T}{\int }}sE(w_{k}(s)-w_{m}(s))ds,\text{\ \ }\forall T\geq \frac{%
3+2\lambda _{1}}{\lambda _{1}}.  \tag{3.4}
\end{equation}%
Since for every $\varepsilon >0$ the embedding $H^{1}(\Omega )\subset
H^{1-\varepsilon }(\Omega )$ is compact (see for example \cite[Theorem 16.1]%
{12}), applying \cite[Corollary 1]{18} to (3.2), we have
\begin{equation*}
w_{k}\rightarrow w\text{\ \ strongly in }C([0,T];H^{1-\varepsilon }(\Omega
)).
\end{equation*}%
Applying Lemma 2.1 it yields that
\begin{equation*}
\left\{
\begin{array}{c}
\sigma (w_{k})\rightarrow \sigma (w)\text{ \ strongly in }C([0,T];L_{\frac{3%
}{2}-\varepsilon }(\Omega )), \\
\sigma ^{\frac{1}{2}}(w_{k})\rightarrow \sigma ^{\frac{1}{2}}(w)\text{ \
strongly in }C([0,T];L_{3-\varepsilon }(\Omega )),%
\end{array}%
\right.
\end{equation*}%
for small enough $\varepsilon >0$. The last approximation together with
(2.3) and (3.2)$_{2}$ implies that%
\begin{equation*}
\left\{
\begin{array}{c}
\sigma (w_{k})w_{kt}\rightarrow \sigma (w)w_{t}\text{ \ weakly in }%
L_{2}([0,T];L_{\frac{6}{5}}(\Omega )), \\
\sigma ^{\frac{1}{2}}(w_{k})w_{kt}\rightarrow \sigma ^{\frac{1}{2}}(w)w_{t}%
\text{ \ weakly in }L_{2}([0,T];L_{2}(\Omega )),%
\end{array}%
\right.
\end{equation*}%
by which we obtain
\begin{equation*}
\underset{m\rightarrow \infty }{\lim \inf }\underset{k\rightarrow \infty }{%
\lim \inf }\underset{0}{\overset{T}{\int }}s\left\langle \sigma
(w_{k}(s))w_{kt}(s)-\sigma
(w_{m}(s))w_{mt}(s),w_{kt}(s)-w_{mt}(s)\right\rangle ds=
\end{equation*}%
\begin{equation*}
=\underset{k\rightarrow \infty }{\lim \inf }\underset{0}{\overset{T}{\int }}%
s\left\Vert \sigma ^{\frac{1}{2}}(w_{k}(s))w_{kt}(s)\right\Vert
_{L_{2}(\Omega )}^{2}ds+\underset{m\rightarrow \infty }{\lim \inf }\underset{%
0}{\overset{T}{\int }}s\left\Vert \sigma ^{\frac{1}{2}}(w_{m}(s))w_{mt}(s)%
\right\Vert _{L_{2}(\Omega )}^{2}ds-
\end{equation*}%
\begin{equation}
-2\underset{0}{\overset{T}{\int }}s\left\Vert \sigma ^{\frac{1}{2}%
}(w(s))w_{t}(s)\right\Vert _{L_{2}(\Omega )}^{2}ds\geq 0,  \tag{3.5}
\end{equation}%
\begin{equation*}
\underset{m\rightarrow \infty }{\lim }\underset{k\rightarrow \infty }{\lim }%
\underset{0}{\overset{T}{\int }}s\left\langle \sigma
(w_{k}(s))w_{kt}(s),w_{m}(s)\right\rangle ds=\underset{0}{\overset{T}{\int }}%
s\left\langle \sigma (w(s))w_{t}(s),w(s)\right\rangle ds=
\end{equation*}%
\begin{equation}
=T\underset{0}{\overset{T}{\int }}\left\langle \widehat{\Sigma }%
(w(s)),1\right\rangle ds-\underset{0}{\overset{T}{\int }}\left\langle
\widehat{\Sigma }(w(s)),1\right\rangle ds  \tag{3.6}
\end{equation}%
and
\begin{equation*}
\underset{m\rightarrow \infty }{\lim }\underset{k\rightarrow \infty }{\lim }%
\underset{0}{\overset{T}{\int }}s\left\langle \sigma
(w_{m}(s))w_{mt}(s),w_{k}(s)\right\rangle ds=\underset{0}{\overset{T}{\int }}%
s\left\langle \sigma (w(s))w_{t}(s),w(s)\right\rangle ds=
\end{equation*}%
\begin{equation}
=T\underset{0}{\overset{T}{\int }}\left\langle \widehat{\Sigma }%
(w(s)),1\right\rangle ds-\underset{0}{\overset{T}{\int }}\left\langle
\widehat{\Sigma }(w(s)),1\right\rangle ds  \tag{3.7}
\end{equation}%
Also applying Fatou's lemma and using (2.1), (2.2), (2.3), (3.2), we have%
\begin{equation}
\left\{
\begin{array}{c}
\underset{k\rightarrow \infty }{\lim \inf }\left\langle \widehat{\Sigma }%
(w_{k}(T)),1\right\rangle \geq \left\langle \widehat{\Sigma }%
(w(T)),1\right\rangle , \\
\underset{k\rightarrow \infty }{\lim \inf }\left\langle
F(w_{k}(T)),1\right\rangle \geq \left\langle F(w(T)),1\right\rangle , \\
\underset{k\rightarrow \infty }{\lim \inf }\underset{0}{\overset{T}{\int }}%
s\left\langle f(w_{k}(s)),w_{k}(s)\right\rangle ds\geq \underset{0}{\overset{%
T}{\int }}s\left\langle f(w(s)),w(s)\right\rangle ds.%
\end{array}%
\right.  \tag{3.8}
\end{equation}%
Taking into account (3.4)-(3.8) in (3.3), we obtain%
\begin{equation*}
\frac{T}{2}\underset{m\rightarrow \infty }{\lim \inf }\underset{k\rightarrow
\infty }{\lim \inf }E(w_{k}(T)-w_{m}(T))+\frac{\lambda _{1}}{2}\underset{%
m\rightarrow \infty }{\lim \inf }\underset{k\rightarrow \infty }{\lim \inf }%
\underset{0}{\overset{T}{\int }}sE(w_{k}(s)-w_{m}(s))ds\leq c_{1}+
\end{equation*}%
\begin{equation}
+2\underset{k\rightarrow \infty }{\lim \inf }\underset{0}{\overset{T}{\int }}%
\left\langle F(w_{k}(s))+\frac{\lambda _{1}}{2}\widehat{\Sigma }%
(w_{k}(s))-F(w(s))-\frac{\lambda _{1}}{2}\widehat{\Sigma }%
(w(s)),1\right\rangle ds,\text{ }  \tag{3.9}
\end{equation}%
for $T\geq \frac{3+2\lambda _{1}}{\lambda _{1}}$. Now let us estimate the
right hand side of (3.9). By (2.1),\ (3.1)$_{1}$ and (3.2), we find that%
\begin{equation*}
\underset{0}{\overset{T}{\int }}\left\vert \left\langle
F(w_{m}(s))-F(w(s)),1\right\rangle \right\vert ds\leq c_{2}\underset{0}{%
\overset{T}{\int }}\left\Vert w_{m}(s)-w(s)\right\Vert _{H_{0}^{1}(\Omega
)}ds\leq c_{3}+c_{4}(\varepsilon )\log (T)+
\end{equation*}%
\begin{equation*}
+\varepsilon \underset{1}{\overset{T}{\int }}s\left\Vert
w_{m}(s)-w(s)\right\Vert _{H_{0}^{1}(\Omega )}^{2}ds\leq
c_{3}+c_{4}(\varepsilon )\log (T)+
\end{equation*}%
\begin{equation}
+\varepsilon \underset{k\rightarrow \infty }{\lim \inf }\underset{0}{\overset%
{T}{\int }}s\left\Vert w_{m}(s)-w_{k}(s)\right\Vert _{H_{0}^{1}(\Omega
)}^{2}ds,\text{ \ \ \ }\forall T\geq 1,\text{ \ }\forall \varepsilon >0.
\tag{3.10}
\end{equation}%
By the same way, we have
\begin{equation*}
\underset{0}{\overset{T}{\int }}\left\vert \left\langle \widehat{\Sigma }%
(w_{m}(s))-\widehat{\Sigma }(w(s)),1\right\rangle \right\vert ds\leq
c_{5}+c_{6}(\varepsilon )\log (T)+
\end{equation*}%
\begin{equation}
+\varepsilon \underset{k\rightarrow \infty }{\lim \inf }\underset{0}{\overset%
{T}{\int }}s\left\Vert w_{m}(s)-w_{k}(s)\right\Vert _{H_{0}^{1}(\Omega
)}^{2}ds,\text{ \ \ \ }\forall T\geq 1,\text{ \ }\forall \varepsilon >0.
\tag{3.11}
\end{equation}%
Now, choosing $\varepsilon $ small enough, by (3.9)-(3.11), we obtain%
\begin{equation*}
\underset{m\rightarrow \infty }{\lim \inf }\underset{k\rightarrow \infty }{%
\lim \inf }E(w_{k}(T)-w_{m}(T))\leq \frac{c_{7}(1+\log (T))}{T},\text{ \ }%
\forall T\geq \max \left\{ 1,\frac{3+2\lambda _{1}}{\lambda _{1}}\right\} .
\end{equation*}%
Choosing $T=T_{0}$ in the last inequality we find%
\begin{equation*}
\underset{n\rightarrow \infty }{\lim \inf }\underset{m\rightarrow \infty }{%
\lim \inf }\left\Vert S(t_{n})\varphi _{n}-S(t_{m})\varphi _{m}\right\Vert _{%
\mathcal{H}}\leq c_{8}\sqrt{\frac{(1+\log (T_{0}))}{T_{0}}},\text{ \ }
\end{equation*}%
and passing to the limit as $T_{0}\rightarrow \infty $ we have%
\begin{equation*}
\underset{n\rightarrow \infty }{\lim \inf }\underset{m\rightarrow \infty }{%
\lim \inf }\left\Vert S(t_{n})\varphi _{n}-S(t_{m})\varphi _{m}\right\Vert _{%
\mathcal{H}}=0.
\end{equation*}%
Similarly one can show that%
\begin{equation}
\underset{k\rightarrow \infty }{\lim \inf }\underset{m\rightarrow \infty }{%
\lim \inf }\left\Vert S(t_{n_{k}})\varphi _{n_{k}}-S(t_{n_{m}})\varphi
_{n_{m}}\right\Vert _{\mathcal{H}}=0,  \tag{3.12}
\end{equation}%
for every subsequence $\left\{ n_{k}\right\} _{k=1}^{\infty }$. Now if the
sequence $\left\{ S(t_{n})\varphi _{n}\right\} _{n=1}^{\infty }$ has no
convergent subsequence in $\mathcal{H}$, then there exist $\varepsilon
_{0}>0 $ and a subsequence $\left\{ n_{k}\right\} _{k=1}^{\infty }$, such
that
\begin{equation*}
\left\Vert S(t_{n_{k}})\varphi _{n_{k}}-S(t_{n_{m}})\varphi
_{n_{m}}\right\Vert _{\mathcal{H}}\geq \varepsilon _{0},\text{ \ }k\neq m.
\end{equation*}%
The last inequality contradicts (3.12).
\end{proof}

Now since, by (2.4), the problem (1.1)-(1.3) has a strict Lyapunov function $%
L(w(t)):=E(w(t))+ \left\langle F(w(t)),1\right\rangle-\left\langle
g,w(t)\right\rangle$, according to [4, Corollary 2.29] we have the following
theorem:

\begin{theorem}
Under conditions (2.1)-(2.3), the semigroup $\{S(t)\}_{t\geq 0}$ possesses a
global attractor $\mathcal{A}_{\mathcal{H}}$ in $\mathcal{H}$.
\end{theorem}

\section{Existence of the global attractor in $\mathcal{H}_{1}$}

To prove the existence of a global attractor in $\mathcal{H}_{1}$ we need
the following lemmas:

\begin{lemma}
Let conditions (2.1)-(2.3) hold and $B$ be a bounded subset of $\mathcal{H}%
_{1}$. Then%
\begin{equation}
\underset{t\geq 0}{\sup }\underset{\varphi \in B}{\sup }\left\Vert
S(t)\varphi \right\Vert _{\mathcal{H}_{1}}<\infty .  \tag{4.1}
\end{equation}
\end{lemma}

\begin{proof}
We use the formal estimates which can be justified by Galerkin's
approximations. Multiplying both sides of (1.1) by $-\Delta w_{t}$ and
integrating over $\Omega $, we obtain%
\begin{equation*}
\frac{d}{dt}\left( \frac{1}{2}\left\Vert \nabla w_{t}(t)\right\Vert
_{L_{2}(\Omega )}^{2}+\frac{1}{2}\left\Vert \Delta w(t)\right\Vert
_{L_{2}(\Omega )}^{2}+\left\langle g,\Delta w(t)\right\rangle \right) +
\end{equation*}%
\begin{equation*}
+\frac{1}{2}\left\Vert \Delta w_{t}(t)\right\Vert _{L_{2}(\Omega )}^{2}\leq
\left\Vert f(w(t))\right\Vert _{L_{2}(\Omega )}^{2}+
\end{equation*}%
\begin{equation}
+\left\Vert \sigma (w(t))w_{t}(t)\right\Vert _{L_{2}(\Omega )}^{2},\text{ \ }%
\forall t\geq 0.  \tag{4.2}
\end{equation}%
By (2.1) and (2.3), we have
\begin{equation*}
\left\Vert f(w(t))\right\Vert _{L_{2}(\Omega )}^{2}+\left\Vert \sigma
(w(t))w_{t}(t)\right\Vert _{L_{2}(\Omega )}^{2}\leq c_{1}\left( 1+\left\Vert
w(t)\right\Vert _{L_{10}(\Omega )}^{10}+\left\Vert w_{t}(t)\right\Vert
_{L_{2}(\Omega )}^{2}\right) +
\end{equation*}%
\begin{equation}
+c_{2}\left\Vert w(t)\right\Vert _{L_{10}(\Omega )}^{8}\left\Vert
w_{t}(t)\right\Vert _{L_{10}(\Omega )}^{2},\text{ }\forall t\geq 0.
\tag{4.3}
\end{equation}%
On the other hand, by the embedding and interpolation theorems, we find%
\begin{equation}
\left\Vert \varphi \right\Vert _{L_{10}(\Omega )}\leq c_{2}\left\Vert
\varphi \right\Vert _{H^{\frac{6}{5}}(\Omega )}\leq c_{3}\left\Vert \varphi
\right\Vert _{H^{2}(\Omega )}^{\frac{1}{5}}\left\Vert \varphi \right\Vert
_{H^{1}(\Omega )}^{\frac{4}{5}},\text{ \ }\forall \varphi \in H^{2}(\Omega ).
\tag{4.4}
\end{equation}%
Taking into account (2.4), (4.3) and (4.4) in (4.2) and applying Gronwall's
lemma, we obtain%
\begin{equation}
\left\Vert (w(t),w_{t}(t))\right\Vert _{\mathcal{H}_{1}}\leq
C(t,r)(1+\left\Vert (w_{0},w_{1})\right\Vert _{\mathcal{H}_{1}}),\text{ \ }%
\forall t\geq 0,  \tag{4.5}
\end{equation}%
where $C:R_{+}\times R_{+}\rightarrow R_{+}$ is a nondecreasing function
with respect to each variable and $r=\underset{\varphi \in B}{\sup }%
\left\Vert \varphi \right\Vert _{\mathcal{H}}$. Since the embedding $%
\mathcal{H}_{1}\subset \mathcal{H}$ is compact, by (4.5), it follows that
the set $\underset{0\leq t\leq T}{\cup }S(t)B$ is a relatively compact
subset of $\mathcal{H}$, for every $T>0$. This together with Lemma 3.1
implies the relative compactness of $\underset{t\geq 0}{\cup }S(t)B$ in $%
\mathcal{H}$. Now using this fact let us estimate $\left\Vert
w(t)\right\Vert _{L_{10}(\Omega )}$:%
\begin{equation*}
\left\Vert w(t)\right\Vert _{L_{10}(\Omega )}^{10}\leq m^{10}mes(\Omega )+%
\underset{\left\{ x:x\in \Omega ,\text{ }\left\vert w(t,x)\right\vert
>m\right\} }{\int }\left\vert w(t,x)\right\vert ^{10}dx\leq
\end{equation*}%
\begin{equation*}
\leq m^{10}mes(\Omega )+\left( \underset{\left\{ x:x\in \Omega ,\text{ }%
\left\vert w(t,x)\right\vert >m\right\} }{\int }\left\vert w(t,x)\right\vert
^{6}dx\right) ^{\frac{1}{3}}\left\Vert w(t)\right\Vert _{L_{12}(\Omega
)}^{8}\leq
\end{equation*}%
\begin{equation*}
\leq m^{10}mes(\Omega )+c_{4}\left( \underset{\left\{ x:x\in \Omega ,\text{ }%
\left\vert w(t,x)\right\vert >m\right\} }{\int }\left\vert w(t,x)\right\vert
^{6}dx\right) ^{\frac{1}{3}}\left\Vert w(t)\right\Vert _{H^{2}(\Omega
)}^{2}\left\Vert w(t)\right\Vert _{H^{1}(\Omega )}^{6}.
\end{equation*}%
So for any $\varepsilon >0$ there exists $c_{\varepsilon }>0$ such that%
\begin{equation*}
\left\Vert w(t)\right\Vert _{L_{10}(\Omega )}\leq \varepsilon \left\Vert
\Delta w(t)\right\Vert _{L_{2}(\Omega )}^{\frac{1}{5}}+c_{\varepsilon },%
\text{ }\forall t\geq 0,
\end{equation*}%
which together with (4.2)-(4.4) yields%
\begin{equation*}
\frac{d}{dt}\left( \frac{1}{2}\left\Vert \nabla w_{t}(t)\right\Vert
_{L_{2}(\Omega )}^{2}+\frac{1}{2}\left\Vert \Delta w(t)\right\Vert
_{L_{2}(\Omega )}^{2}+\left\langle g,\Delta w(t)\right\rangle \right) +\frac{%
1}{4}\left\Vert \Delta w_{t}(t)\right\Vert _{L_{2}(\Omega )}^{2}\leq
\end{equation*}%
\begin{equation*}
\leq c_{5}\left\Vert \nabla w_{t}(t)\right\Vert _{L_{2}(\Omega
)}^{2}\left\Vert \Delta w(t)\right\Vert _{L_{2}(\Omega )}^{2}+\varepsilon
\left\Vert \Delta w(t)\right\Vert _{L_{2}(\Omega )}^{2}+\widetilde{c}%
_{\varepsilon }+c_{5},\text{ \ }\forall t\geq 0.
\end{equation*}%
Now multiplying both sides of (1.1) by $-\mu \Delta w$ ($\mu \in (0,1)$) and
integrating over $\Omega $, we obtain%
\begin{equation*}
\frac{d}{dt}\left( \frac{1}{2}\mu \left\Vert \Delta w(t)\right\Vert
_{L_{2}(\Omega )}^{2}+\mu \left\langle \nabla w_{t}(t),\nabla
w(t)\right\rangle \right) +\mu \left\Vert \Delta w(t)\right\Vert
_{L_{2}(\Omega )}^{2}\leq
\end{equation*}%
\begin{equation*}
\leq \mu \left\Vert g\right\Vert _{L_{2}(\Omega )}\left\Vert \Delta
w(t)\right\Vert _{L_{2}(\Omega )}+\mu \left\Vert \nabla w_{t}(t)\right\Vert
_{L_{2}(\Omega )}^{2}+\mu \left\Vert \sigma (w(t))w_{t}(t)\right\Vert
_{L_{2}(\Omega )}\left\Vert \Delta w(t)\right\Vert _{L_{2}(\Omega )}
\end{equation*}%
\begin{equation*}
+\mu \left\Vert f(w(t))\right\Vert _{L_{2}(\Omega )}\left\Vert \Delta
w(t)\right\Vert _{L_{2}(\Omega )},\text{ \ \ }\forall t\geq 0.
\end{equation*}%
Taking into account the relative compactness of $\underset{t\geq 0}{\cup }%
S(t)B$, similar to the argument done above, we can say that for any $%
\varepsilon >0$ there exists $\widehat{c}_{\varepsilon }>0$ such that%
\begin{equation*}
\left\Vert f(w(t))\right\Vert _{L_{2}(\Omega )}^{2}+\left\Vert \sigma
(w(t))w_{t}(t)\right\Vert _{L_{2}(\Omega )}^{2}\leq \varepsilon \left(
\left\Vert \Delta w(t)\right\Vert _{L_{2}(\Omega )}^{2}+\left\Vert \Delta
w_{t}(t)\right\Vert _{L_{2}(\Omega )}^{2}\right) +
\end{equation*}%
\begin{equation*}
+\widehat{c}_{\varepsilon }\left\Vert \Delta w(t)\right\Vert _{L_{2}(\Omega
)}^{2}\left\Vert \nabla w_{t}(t)\right\Vert _{L_{2}(\Omega )}^{2}+\widehat{c}%
_{\varepsilon },\text{ }\forall t\geq 0.
\end{equation*}%
By the last three inequalities we have%
\begin{equation*}
\frac{d}{dt}\left( \frac{1}{2}\left\Vert \nabla w_{t}(t)\right\Vert
_{L_{2}(\Omega )}^{2}+\frac{1}{2}(1+\mu )\left\Vert \Delta w(t)\right\Vert
_{L_{2}(\Omega )}^{2}+\mu \left\langle \nabla w_{t}(t),\nabla
w(t)\right\rangle +\left\langle g,\Delta w(t)\right\rangle \right)
\end{equation*}%
\begin{equation*}
+(\frac{1}{4}-\mu c_{6}-\varepsilon )\left\Vert \Delta w_{t}(t)\right\Vert
_{L_{2}(\Omega )}^{2}+(\frac{1}{4}\mu -2\varepsilon )\left\Vert \Delta
w(t)\right\Vert _{L_{2}(\Omega )}^{2}\leq
\end{equation*}%
\begin{equation*}
\leq (c_{5}+\widehat{c}_{\varepsilon })\left\Vert \Delta w(t)\right\Vert
_{L_{2}(\Omega )}^{2}\left\Vert \nabla w_{t}(t)\right\Vert _{L_{2}(\Omega
)}^{2}+c_{6}+\widehat{c}_{\varepsilon }+\widetilde{c}_{\varepsilon },\text{ }%
\forall t\geq 0.
\end{equation*}%
Choosing $\mu $ small enough and $\varepsilon \in (0,\frac{1}{8}\mu ),$ we
obtain%
\begin{equation*}
\frac{d}{dt}\Phi (t)+c_{7}\Phi (t)\leq c_{8}\left\Vert \nabla
w_{t}(t)\right\Vert _{L_{2}(\Omega )}^{2}\Phi (t)+c_{8}(1+\left\Vert \nabla
w_{t}(t)\right\Vert _{L_{2}(\Omega )}^{2}),\text{ }\forall t\geq 0,
\end{equation*}%
where $\Phi (t)=\frac{1}{2}\left\Vert \nabla w_{t}(t)\right\Vert
_{L_{2}(\Omega )}^{2}+\frac{1}{2}(1+\mu )\left\Vert \Delta w(t)\right\Vert
_{L_{2}(\Omega )}^{2}+\mu \left\langle \nabla w_{t}(t),\nabla
w(t)\right\rangle +\newline
+\left\langle g,\Delta w(t)\right\rangle $. Multiplying both sides of the
last inequality by\newline
$e^{\underset{0}{\overset{t}{\int }}(c_{7}-c_{8}\left\Vert \nabla w_{t}(\tau
)\right\Vert _{L_{2}(\Omega )}^{2})d\tau }$ , integrating over $\left[ 0,T%
\right] $ and multiplying both sides of obtained inequality by $e^{-\underset%
{0}{\overset{T}{\int }}\left[ c_{7}-c_{8}\left\Vert \nabla
w_{t}(t)\right\Vert _{L_{2}(\Omega )}^{2}\right] dt}$, we find%
\begin{equation*}
\Phi (T)\leq \Phi (0)e^{-\underset{0}{\overset{T}{\int }}(c_{7}-c_{8}\left%
\Vert \nabla w_{t}(t)\right\Vert _{L_{2}(\Omega )}^{2})dt}+
\end{equation*}%
\begin{equation*}
+c_{8}\underset{0}{\overset{T}{\int }}(1+\left\Vert \nabla
w_{t}(t)\right\Vert _{L_{2}(\Omega )}^{2})e^{-\underset{t}{\overset{T}{\int }%
}(c_{7}-c_{8}\left\Vert \nabla w_{t}(\tau )\right\Vert _{L_{2}(\Omega
)}^{2})d\tau }dt,\text{ }\forall T\geq 0,
\end{equation*}%
which together with (2.4) yields (4.1).
\end{proof}

\begin{lemma}
Let conditions (2.1)-(2.3) hold and $B$ be\ a bounded subset of $\mathcal{H}%
_{1}$. Then every sequence of the form $\left\{ S(t_{n})\varphi _{n}\right\}
_{n=1}^{\infty }$, $\left\{ \varphi _{n}\right\} _{n=1}^{\infty }\subset B$,
$t_{n}\rightarrow \infty $, has a convergent subsequence in $\mathcal{H}_{1}$%
.
\end{lemma}

\begin{proof}
Let us decompose $\left\{ S(t)\right\} _{t\geq 0}$ as $S(t)=U(t)+C(t)$,
where $U(t)$ is a linear semigroup generated by the problem%
\begin{equation}
\left\{
\begin{array}{c}
u_{tt}-\Delta u_{t}-\Delta u=0,\text{ \ \ in \ \ }(0,\infty )\times \Omega ,
\\
u=0,\text{ \ \ \ \ \ \ \ \ \ \ \ \ \ \ \ \ \ \ \ on \ }(0,\infty )\times
\partial \Omega , \\
u(0,\cdot )=w_{0}\text{ },\text{ \ \ }u_{t}(0,\cdot )=w_{1},\text{\ \ \ \ in
\ }\Omega ,%
\end{array}%
\right.  \tag{4.6}
\end{equation}%
$C(t)$ is a solution operator of
\begin{equation}
\left\{
\begin{array}{c}
v_{tt}-\Delta v_{t}-\Delta v=g(x)-f(w)-\sigma (w)w_{t},\text{ \ in \ }%
(0,\infty )\times \Omega , \\
v=0,\text{ \ \ \ \ \ \ \ \ \ \ \ \ \ \ \ \ \ \ \ \ \ \ \ \ \ \ \ \ \ \ \ \ \
\ \ \ \ \ \ \ \ \ \ on \ }(0,\infty )\times \partial \Omega , \\
v_{k}(0,\cdot )=0\text{ },\text{ \ \ \ \ \ \ }v_{t}(0,\cdot )=0,\text{\ \ \
\ \ \ \ \ \ \ \ \ \ \ \ \ \ \ \ \ \ \ \ \ \ \ \ \ \ in \ }\Omega%
\end{array}%
\right.  \tag{4.7}
\end{equation}%
(i.e. $(u(t),u_{t}(t))=U(t)(w_{0},w_{1})$ and $%
(v(t),v_{t}(t))=C(t)(w_{0},w_{1})$) and $(w(t), w_{t}(t))$= $\
S(t)(w_{0},w_{1})$. Multiplying (4.6)$_{1}$ by $(u_{t}-\frac{1}{2}\Delta
u-\mu \Delta u_{t}-\nu t\Delta u_{t})$ and integrating over $\Omega $, we
obtain%
\begin{equation*}
\frac{d}{dt}\left( E(u(t))+\frac{1}{4}\left\Vert \Delta u(t)\right\Vert
_{L_{2}(\Omega )}^{2}-\frac{1}{2}\left\langle u_{t},\Delta u\right\rangle +%
\frac{1}{2}(\mu +\nu t)\left\Vert \nabla u_{t}(t)\right\Vert _{L_{2}(\Omega
)}^{2}+\right.
\end{equation*}%
\begin{equation*}
\left. +\frac{1}{2}(\mu +\nu t)\left\Vert \Delta u(t)\right\Vert
_{L_{2}(\Omega )}^{2}\right) +\frac{1}{2}(1-\nu )\left\Vert \nabla
u_{t}(t)\right\Vert _{L_{2}(\Omega )}^{2}+\frac{1}{2}(1-\nu )\left\Vert
\Delta u(t)\right\Vert _{L_{2}(\Omega )}^{2}+
\end{equation*}%
\begin{equation*}
+(\mu +\nu t)\left\Vert \Delta u_{t}(t)\right\Vert _{L_{2}(\Omega )}^{2}=0,%
\text{ \ \ \ \ }\forall t\geq 0.
\end{equation*}%
Choosing $(\mu ,\nu )=(1,0)$ and $(\mu ,\nu )=(0,1)$ in the last equality,
we find
\begin{equation}
\left\Vert U(t)\right\Vert _{\mathcal{L}(\mathcal{H}_{1},\mathcal{H}%
_{1})}\leq Me^{-\omega t},\text{ \ \ \ }\forall t\geq 0,  \tag{4.8}
\end{equation}%
and
\begin{equation}
\left\Vert U(t)\right\Vert _{\mathcal{L}((H^{2}(\Omega )\cap
H_{0}^{1}(\Omega ))\times L_{2}(\Omega ),\mathcal{H}_{1})}\leq \frac{M}{%
\sqrt{t}},\text{ \ }\forall t>0,  \tag{4.9}
\end{equation}%
respectively, where $M>0$ and $\omega >0$. \ Also applying Duhamel's
principle to (4.7), we have
\begin{equation}
C(t)(w_{0},w_{1})=\underset{0}{\overset{t}{\int }}U(t-s)(0,\Phi
_{(w_{0},w_{1})}(s))ds,  \tag{4.10}
\end{equation}%
where $\Phi _{(w_{0},w_{1})}(s)=g-f(w(s))-\sigma (w(s))w_{t}(s)$. By Lemma
4.1 and equation (1.1), it follows that the set of functions $\left\{ \Phi
_{(w_{0},w_{1})}(s):(w_{0},w_{1})\in B\text{ \ }\right\} $ is precompact in $%
C([0,t];L_{2}(\Omega ))$. So, from (4.9) and (4.10) we obtain that the
operator $C(t):\mathcal{H}_{1}\rightarrow \mathcal{H}_{1}$, $t\geq 0$, is
compact. Since
\begin{equation*}
S(t_{n})\varphi _{n}=U(T)S(t_{n}-T)\varphi _{n}+C(T)S(t_{n}-T)\varphi _{n}
\end{equation*}%
for $t_{n}\geq T$, by (4.1), (4.8) and the compactness of $C(t)$, we obtain
that the sequence $\left\{ S(t_{n})\varphi _{n}\right\} _{n=1}^{\infty }$
has a finite $\varepsilon $-net in $\mathcal{H}$, for every $\varepsilon >0$%
. This completes the proof.
\end{proof}

Now by Lemma 4.2, similar to Theorem 3.1, we obtain the following theorem:

\begin{theorem}
Under conditions (2.1)-(2.3), the semigroup $\{S(t)\}_{t\geq 0}$ possesses a
global attractor $\mathcal{A}_{\mathcal{H}_{1}}$ in $\mathcal{H}_{1}$.
\end{theorem}

\section{Regularity of the $\mathcal{A}_{\mathcal{H}}$}

To prove the regularity of $\mathcal{A}_{\mathcal{H}}$ we will use the
method used in \cite{9} and \cite{10}. Since $\mathcal{A}_{\mathcal{H}}$ is
invariant, by \cite[p. 159]{1}, for every $(w_{0},w_{1})\in \mathcal{A}_{%
\mathcal{H}}$ there exists an invariant trajectory $\gamma
=\{W(t)=(w(t),w_{t}(t)),$ $t\in R\}\subset $ $\mathcal{A}_{\mathcal{H}}$
such that $W(0)=(w_{0},w_{1})$. By an invariant trajectory we mean a curve $%
\gamma =\{W(t),$ $t\in R\}$ such that $\ S(t)W(\tau )=W(t+\tau )$ for $t\geq
0$ and $\tau \in R$ (see \cite[p. 157]{1}). Let us decompose $w(t)$ as $%
w(t)=u_{k}(t,s)+v_{k}(t,s)$, where%
\begin{equation}
\left\{
\begin{array}{c}
v_{ktt}-\Delta v_{kt}+\sigma _{k}(w)v_{kt}-\Delta v_{k}+f_{k}(w)=g(x),\text{
\ \ in \ \ }(s,\infty )\times \Omega , \\
v_{k}=0,\text{ \ \ \ \ \ \ \ \ \ \ \ \ \ \ \ \ \ \ \ \ \ \ \ \ \ \ \ \ \ \ \
\ \ \ \ on \ }(s,\infty )\times \partial \Omega , \\
v_{k}(s,s,\cdot )=0\text{ },\text{ \ \ \ \ \ \ }v_{kt}(s,s,\cdot )=0,\text{\
\ \ \ \ \ \ \ \ \ \ \ \ \ \ in \ }\Omega%
\end{array}%
\right. ,  \tag{5.1}
\end{equation}

\begin{equation}
\left\{
\begin{array}{c}
u_{ktt}-\Delta u_{kt}+\sigma (w)w_{t}-\sigma _{k}(w)v_{kt}-\Delta u_{k}=%
\text{ \ \ \ \ \ \ \ \ \ \ \ \ \ \ \ } \\
=f_{k}(w)-f(w),\text{ \ \ \ \ \ \ \ \ \ \ \ \ \ \ \ \ \ \ \ \ \ \ \ \ \ \ \
\ \ \ \ \ \ \ \ \ in \ \ }(s,\infty )\times \Omega , \\
u_{k}=0,\text{ \ \ \ \ \ \ \ \ \ \ \ \ \ \ \ \ \ \ \ \ \ \ \ \ \ \ \ \ \ \ \
\ \ \ \ \ \ \ \ \ \ \ \ on \ }(s,\infty )\times \partial \Omega , \\
u_{k}(s,s,\cdot )=w(s,\cdot )\text{ },\text{ \ \ \ \ \ }u_{kt}(s,s,\cdot
)=w_{t}(s,\cdot ),\text{\ \ \ \ \ \ \ \ \ \ in \ }\Omega%
\end{array}%
\right. ,  \tag{5.2}
\end{equation}%
$f_{k}(s)=\left\{
\begin{array}{c}
f(k),\text{\ \ \ \ \ \ \ \ \ }s>k, \\
f(s),\text{ \ \ \ \ \ \ }\left\vert s\right\vert \leq k, \\
f(-k),\text{ \ \ \ \ \ \ }s<-k\text{\ }%
\end{array}%
\right. $, $\ \sigma _{k}(s)=\left\{
\begin{array}{c}
\sigma (k),\text{\ \ \ \ \ \ \ \ }s>k, \\
\sigma (s),\text{ \ \ \ \ \ \ }\left\vert s\right\vert \leq k, \\
\sigma (-k),\text{ \ \ \ \ \ \ }s<-k\text{\ }%
\end{array}%
\right. $ and $k\in
\mathbb{N}
$.\newline
Now let us prove the following lemmas:

\begin{lemma}
\textit{Assume that conditions (2.1)-(2.3) are satisfied}. \textit{Then }%
\newline
$(v_{k}(t,s),v_{kt}(t,s))\in \mathcal{H}_{1}$ and for any $k\in
\mathbb{N}
$ there exists $T_{k}<0$ such that
\begin{equation}
\left\Vert v_{kt}(t,s)\right\Vert _{H^{1}(\Omega )}+\left\Vert
v_{k}(t,s)\right\Vert _{H^{2}(\Omega )}\leq r_{0}k^{\frac{128}{65}},\text{ }%
\forall s\leq t\leq T_{k},  \tag{5.3}
\end{equation}%
where the positive constant $r_{0}$ is independent of $k$ and $(w_{0},w_{1})$%
.
\end{lemma}

\begin{proof}
Multiplying both sides of (5.1)$_{1}$ by $v_{kt}+\mu v_{k}$ ($\mu \in (0,1)$%
) and integrating over $\Omega $, we obtain%
\begin{equation*}
\frac{d}{dt}\left( E(v_{k}(t,s))+\frac{\mu }{2}\left\Vert \nabla
v_{k}(t,s)\right\Vert _{L_{2}(\Omega )}^{2}+\mu \left\langle
v_{kt}(t,s),v_{k}(t,s)\right\rangle \right) +
\end{equation*}%
\begin{equation*}
+\frac{1}{2}\left\Vert \nabla v_{kt}(t,s)\right\Vert _{L_{2}(\Omega
)}^{2}-\mu \left\Vert v_{kt}(t,s)\right\Vert _{L_{2}(\Omega )}^{2}+(\mu
-c_{1}\mu ^{2})\left\Vert \nabla v_{k}(t,s)\right\Vert _{L_{2}(\Omega
)}^{2}\leq c_{2},\text{ }\forall t\geq s.
\end{equation*}%
Choosing $\mu $ small enough in the last inequality, we find%
\begin{equation}
\left\Vert v_{kt}(t,s)\right\Vert _{L_{2}(\Omega )}+\left\Vert
v_{k}(t,s)\right\Vert _{H_{0}^{1}(\Omega )}\leq c_{3},\text{ \ \ }\forall
t\geq s.  \tag{5.4}
\end{equation}%
Multiplying both sides of (5.1)$_{1}$ by $v_{kt}$, integrating over $(\tau
_{1},\tau _{2})\times \Omega $ and taking into account (5.4), we have%
\begin{equation*}
\underset{\tau _{1}}{\overset{\tau _{2}}{\int }}\left\Vert \nabla
v_{kt}(t,s)\right\Vert _{L_{2}(\Omega )}^{2}dt\leq c_{4}+\underset{\tau _{1}}%
{\overset{\tau _{2}}{\int }}\left\vert \left\langle f_{k}^{\prime
}(w(t))w_{t}(t),v_{k}(t,s)\right\rangle \right\vert dt\leq c_{4}+
\end{equation*}%
\begin{equation}
+c_{5}\underset{\tau _{1}}{\overset{\tau _{2}}{\int }}\left\Vert \nabla
w_{t}(t)\right\Vert _{L_{2}(\Omega )}dt,\text{ \ \ }\forall \tau _{2}\geq
\tau _{1}\geq s.  \tag{5.5}
\end{equation}%
On the other hand, by (2.4), we have%
\begin{equation}
\underset{-\infty }{\overset{\infty }{\int }}\left\Vert \nabla
w_{t}(t)\right\Vert _{L_{2}(\Omega )}^{2}dt<\infty ,  \tag{5.6}
\end{equation}%
which together with (5.5) yields%
\begin{equation}
\underset{\tau _{1}}{\overset{\tau _{2}}{\int }}\left\Vert \nabla
v_{kt}(t,s)\right\Vert _{L_{2}(\Omega )}^{2}dt\leq c_{6}(1+\left( \tau
_{2}-\tau _{1}\right) ^{\frac{1}{2}}),\text{ \ }\forall \tau _{2}\geq \tau
_{1}\geq s.  \tag{5.7}
\end{equation}%
Multiplying both sides of (5.1)$_{1}$ by $-\Delta v_{kt}-\mu \Delta v_{k}$ ($%
\mu \in (0,1)$), integrating over $\Omega $ and taking into account (5.4),
we have
\begin{equation*}
\frac{d}{dt}\left( \frac{1}{2}\left\Vert \nabla v_{kt}(t,s)\right\Vert
_{L_{2}(\Omega )}^{2}+\frac{1}{2}\left\Vert \Delta v_{k}(t,s)\right\Vert
_{L_{2}(\Omega )}^{2}+\mu \left\langle \nabla v_{kt}(t,s),\nabla
v_{k}(t,s)\right\rangle \right) +
\end{equation*}%
\begin{equation*}
+(\frac{1}{2}-c_{7}\mu )\left\Vert \Delta v_{kt}(t,s)\right\Vert
_{L_{2}(\Omega )}^{2}+(\mu -\mu ^{2})\left\Vert \Delta v_{k}(t,s)\right\Vert
_{L_{2}(\Omega )}^{2}\leq c_{7}+
\end{equation*}%
\begin{equation}
+c_{7}\left\Vert \sigma _{k}(w(t))v_{kt}(t,s)\right\Vert _{L_{2}(\Omega
)}^{2}+c_{7}\left\Vert f_{k}(w(t))\right\Vert _{L_{2}(\Omega )}^{2},\text{ }%
\forall t\geq s.  \tag{5.8}
\end{equation}%
Now let us estimate the last two terms on the right side of (5.8). By \
(4.4) and (5.4), we find%
\begin{equation*}
\left\Vert \sigma _{k}(w(t))v_{kt}(t,s)\right\Vert _{L_{2}(\Omega )}^{2}\leq
\left\Vert \sigma _{k}(w(t))\right\Vert _{L_{\frac{5}{2}}(\Omega
)}^{2}\left\Vert v_{kt}(t,s)\right\Vert _{L_{10}(\Omega )}^{2}\leq
\end{equation*}%
\begin{equation*}
\leq c_{8}\left\Vert \sigma _{k}(w(t))\right\Vert _{L_{\frac{5}{2}}(\Omega
)}^{2}\left\Vert v_{kt}(t,s)\right\Vert _{H^{2}(\Omega )}^{\frac{2}{5}%
}\left\Vert v_{kt}(t,s)\right\Vert _{H^{1}(\Omega )}^{\frac{8}{5}}\leq
\end{equation*}%
\begin{equation*}
\leq c_{9}\left\Vert \sigma _{k}(w(t))\right\Vert _{L_{\frac{5}{2}}(\Omega
)}^{4}+c_{9}\left\Vert \Delta v_{kt}(t,s)\right\Vert _{L_{2}(\Omega
)}^{2}\left\Vert \nabla v_{kt}(t,s)\right\Vert _{L_{2}(\Omega )}^{2}+
\end{equation*}%
\begin{equation}
+\frac{1}{3c_{7}}\left\Vert \Delta v_{kt}(t,s)\right\Vert _{L_{2}(\Omega
)}^{2},\text{ }\forall t\geq s.  \tag{5.9}
\end{equation}%
Also by the definitions of $\sigma _{k}(\cdot )$ and $f_{k}(\cdot )$, we have%
\begin{equation*}
\left\Vert \sigma _{k}(w(t))\right\Vert _{L_{\frac{5}{2}}(\Omega )}^{\frac{5%
}{2}}=\underset{\Omega }{\int }\left\vert \sigma _{k}(w(t,x))\right\vert ^{%
\frac{5}{2}}dx\leq
\end{equation*}%
\begin{equation*}
\leq \underset{\left\{ x:x\in \Omega ,\text{ }\left\vert w(t,x)\right\vert
\leq 2m\right\} }{\int }\left\vert \sigma _{k}(w(t,x))\right\vert ^{\frac{5}{%
2}}dx+\underset{\left\{ x:x\in \Omega ,\text{ }\left\vert w(t,x)\right\vert
>2m\right\} }{\int }\left\vert \sigma _{k}(w(t,x))\right\vert ^{\frac{5}{2}%
}dx\leq
\end{equation*}%
\begin{equation*}
\leq c_{10}m^{4}\underset{\left\{ x:x\in \Omega ,\text{ }\left\vert
w(t,x)\right\vert \leq 2m\right\} }{\int }( 1+\left\vert w(t,x)\right\vert
^{6} )dx+
\end{equation*}%
\begin{equation*}
+c_{10}k^{4}\underset{\left\{ x:x\in \Omega ,\text{ }\left\vert
w(t,x)\right\vert >2m\right\} }{\int }\left\vert w(t,x)\right\vert
^{6}dx\leq c_{11}m^{4}+
\end{equation*}%
\begin{equation}
+c_{10}k^{4}\underset{\left\{ x:x\in \Omega ,\text{ }\left\vert
w(t,x)\right\vert >2m\right\} }{\int }\left\vert w(t,x)\right\vert ^{6}dx,%
\text{ \ \ }\forall k\in
\mathbb{N}
\ \forall m\geq 1 \text{ and}\ \forall t\in R.  \tag{5.10}
\end{equation}%
\begin{equation*}
\left\Vert f_{k}(w(t))\right\Vert _{L_{2}(\Omega )}^{2}=\underset{\Omega }{%
\int }\left\vert f_{k}(w(t,x))\right\vert ^{2}dx\leq
\end{equation*}%
\begin{equation*}
\leq c_{12}m^{4}\underset{\left\{ x:x\in \Omega ,\text{ }\left\vert
w(t,x)\right\vert \leq 2m\right\} }{\int }(1+\left\vert w(t,x)\right\vert
^{6})dx+
\end{equation*}%
\begin{equation*}
+c_{12}k^{4}\underset{\left\{ x:x\in \Omega ,\text{ }\left\vert
w(t,x)\right\vert >2m\right\} }{\int }\left\vert w(t,x)\right\vert
^{6}dx\leq c_{13}m^{4}+
\end{equation*}%
\begin{equation}
c_{12}k^{4}\underset{\left\{ x:x\in \Omega ,\text{ }\left\vert
w(t,x)\right\vert >2m\right\} }{\int }\left\vert w(t,x)\right\vert ^{6}dx,%
\text{ \ \ }\forall k\in
\mathbb{N}
\ \forall m\geq 1\text{ and}\ \forall t\in R.  \tag{5.11}
\end{equation}%
Now denote $w^{(m)}(t,x)=\left\{
\begin{array}{c}
w(t,x)-m,\text{ \ \ \ \ \ }w(t,x)>m \\
0,\text{ \ \ \ \ \ \ \ \ \ \ \ \ \ \ \ \ }\left\vert w(t,x)\right\vert \leq m
\\
w(t,x)+m,\text{ \ \ \ \ \ }w(t,x)<-m%
\end{array}%
\right. $. Since,%
\begin{equation*}
\left\vert w(t,x)\right\vert <2\left\vert w^{(m)}(t,x)\right\vert ,\text{ \
\ \ \ }\forall (t,x)\in \left\{ (t,x)\in R\times \Omega ,\text{ }\left\vert
w(t,x)\right\vert >2m\right\} ,
\end{equation*}

we have
\begin{equation*}
\underset{\left\{ x:x\in \Omega ,\text{ }\left\vert w(t,x)\right\vert
>2m\right\} }{\int }\left\vert w(t,x)\right\vert ^{6}dx\leq 2^{6}\underset{%
\left\{ x:x\in \Omega ,\text{ }\left\vert w(t,x)\right\vert >2m\right\} }{%
\int }\left\vert w^{m}(t,x)\right\vert ^{6}dx\leq
\end{equation*}%
\begin{equation}
\leq 2^{6}\underset{\Omega }{\int }\left\vert w^{m}(t,x)\right\vert
^{6}dx\leq c_{14}\left\Vert \nabla w^{(m)}(t)\right\Vert _{L_{2}(\Omega
)}^{2},\text{ \ \ }\forall t\in R.\text{\ }  \tag{5.12}
\end{equation}%
So, by (5.8)-(5.12), it follows that%
\begin{equation*}
\frac{d}{dt}\left( \frac{1}{2}\left\Vert \nabla v_{kt}(t,s)\right\Vert
_{L_{2}(\Omega )}^{2}+\frac{1}{2}\left\Vert \Delta v_{k}(t,s)\right\Vert
_{L_{2}(\Omega )}^{2}+\mu \left\langle \nabla v_{kt}(t,s),\nabla
v_{k}(t,s)\right\rangle \right) +
\end{equation*}%
\begin{equation*}
+(\frac{1}{6}-c_{7}\mu )\left\Vert \Delta v_{kt}(t,s)\right\Vert
_{L_{2}(\Omega )}^{2}+(\mu -\mu ^{2})\left\Vert \Delta v_{k}(t,s)\right\Vert
_{L_{2}(\Omega )}^{2}\leq c_{15}m^{\frac{32}{5}}+
\end{equation*}%
\begin{equation*}
+c_{15}\left\Vert \Delta v_{kt}(t,s)\right\Vert _{L_{2}(\Omega
)}^{2}\left\Vert \nabla v_{kt}(t,s)\right\Vert _{L_{2}(\Omega )}^{2}+
\end{equation*}%
\begin{equation}
+c_{15}k^{\frac{32}{5}}\left\Vert \nabla w^{(m)}(t)\right\Vert
_{L_{2}(\Omega )}^{2},\text{ \ }\forall k\in
\mathbb{N}
\ \forall m\geq 1\ \text{and}\ \forall t\geq s.  \tag{5.13}
\end{equation}%
On the other hand, testing (1.1) by $w^{(m)}$, we obtain%
\begin{equation*}
\frac{d}{dt}\left\langle w_{t}(t),w^{(m)}(t)\right\rangle +\left\Vert \nabla
w^{(m)}(t)\right\Vert _{L_{2}(\Omega )}^{2}-\left\Vert
w_{t}^{(m)}(t)\right\Vert _{L_{2}(\Omega )}^{2}+\left\langle \nabla
w_{t}(t),\nabla w^{(m)}(t)\right\rangle =
\end{equation*}%
\begin{equation}
=\left\langle g,w^{(m)}(t)\right\rangle -\left\langle \sigma
(w(t))w_{t}(t),w^{(m)}(t)\right\rangle -\left\langle
f(w(t)),w^{(m)}(t)\right\rangle ,\text{ }\forall t\in R.  \tag{5.14}
\end{equation}%
Let us estimate each term on the right hand side of (5.14). By the
definition of $w^{(m)}$, we have
\begin{equation*}
\left\langle g,w^{(m)}(t)\right\rangle \leq \left( \underset{\left\{ x:x\in
\Omega ,\text{ }\left\vert w(t,x)\right\vert >m\right\} }{\int }\left\vert
g(x)\right\vert ^{\frac{6}{5}}dx\right) ^{\frac{5}{6}}\left\Vert
w^{(m)}(t)\right\Vert _{L_{6}(\Omega )}\leq
\end{equation*}%
\begin{equation*}
\leq \frac{c_{16}}{m^{2}}\left\Vert \nabla w^{(m)}(t)\right\Vert
_{L_{2}(\Omega )},\text{ \ \ }\forall t\in R.
\end{equation*}%
By (2.3), it follows that
\begin{equation*}
\left\vert \left\langle \sigma (w(t))w_{t}(t),w^{(m)}(t)\right\rangle
\right\vert \leq c_{17}\left\Vert \nabla w_{t}(t)\right\Vert _{L_{2}(\Omega
)}\left\Vert \nabla w^{(m)}(t)\right\Vert _{L_{2}(\Omega )},\text{ \ }%
\forall t\in R.
\end{equation*}%
Also by (2.3), we obtain
\begin{equation*}
\left\langle f(w(t)),w^{(m)}(t)\right\rangle >-\lambda _{1}\left\langle
w(t),w^{(m)}(t)\right\rangle \geq
\end{equation*}%
\begin{equation*}
\geq -\lambda _{1}\left( \underset{\left\{ x:x\in \Omega ,\text{ }\left\vert
w(t,x)\right\vert >m\right\} }{\int }\left\vert w(t,x)\right\vert ^{\frac{6}{%
5}}dx\right) ^{\frac{5}{6}}\left\Vert w^{(m)}(t)\right\Vert _{L_{6}(\Omega
)}\geq
\end{equation*}%
\begin{equation*}
\geq -\frac{c_{18}}{m^{4}}\left\Vert \nabla w^{(m)}(t)\right\Vert
_{L_{2}(\Omega )},\text{ \ }\forall t\in R,
\end{equation*}%
for large enough $m$. Taking into account the last three inequalities in
(5.14), we have
\begin{equation*}
\frac{d}{dt}\left\langle w_{t}(t),w^{(m)}(t)\right\rangle +c_{19}\left\Vert
\nabla w^{(m)}(t)\right\Vert _{L_{2}(\Omega )}^{2}\leq
\end{equation*}%
\begin{equation}
\leq c_{20}\left\Vert \nabla w_{t}(t)\right\Vert _{L_{2}(\Omega )}^{2}+\frac{%
c_{20}}{m^{4}},\text{ \ }\forall t\in R.  \tag{5.15}
\end{equation}%
for large enough $m$. Now multiplying (5.15) by $\frac{c_{15}}{c_{19}}k^{%
\frac{32}{5}}$, adding to (5.13) and then choosing $m=k^{\frac{8}{13}}$, we
get%
\begin{equation*}
\frac{d}{dt}\Lambda _{k,s}(t)+\widehat{c}_{1}\Lambda _{k,s}(t)\leq \widehat{c%
}_{2}\Lambda _{k,s}(t)\left\Vert \nabla v_{kt}(t,s)\right\Vert
_{L_{2}(\Omega )}^{2}+
\end{equation*}%
\begin{equation*}
+\widehat{c}_{2}k^{\frac{256}{65}}+\widehat{c}_{2}k^{\frac{32}{5}}\left\Vert
\nabla w_{t}(t)\right\Vert _{L_{2}(\Omega )}^{2}+\widehat{c}_{2}k^{\frac{64}{%
5}}\left\vert \left\langle w_{t}(t),w^{(k^{\frac{13}{8}})}(t)\right\rangle
\right\vert ^{2},\text{ \ }\forall t\geq s,
\end{equation*}%
for large enough $k$ and small enough $\mu $, where $\widehat{c}_{1}$ and $%
\widehat{c}_{2}$ are positive constants and $\Lambda _{k,s}(t):=\frac{1}{2}%
\left\Vert \nabla v_{kt}(t,s)\right\Vert _{L_{2}(\Omega )}^{2}+\frac{1}{2}%
\left\Vert \Delta v_{k}(t,s)\right\Vert _{L_{2}(\Omega )}^{2}+\mu
\left\langle \nabla v_{kt}(t,s),\nabla v_{k}(t,s)\right\rangle +\frac{c_{15}%
}{c_{16}}k^{\frac{32}{5}}\left\langle w_{t}(t),w^{(k^{\frac{8}{13}%
})}(t)\right\rangle $. Since
\begin{equation*}
\left\vert \left\langle w_{t}(t),w^{(k^{\frac{8}{13}})}(t)\right\rangle
\right\vert \leq \left\Vert w_{t}(t)\right\Vert _{L_{6}(\Omega )}\left(
\underset{\left\{ x:x\in \Omega ,\text{ }\left\vert w(t,x)\right\vert >k^{%
\frac{8}{13}}\right\} }{\int }\left\vert w(t,x)\right\vert ^{\frac{6}{5}%
}dx\right) ^{\frac{5}{6}}\leq
\end{equation*}%
\begin{equation*}
\leq \frac{\widehat{c}_{3}}{k^{\frac{32}{13}}}\left\Vert \nabla
w_{t}(t)\right\Vert _{L_{2}(\Omega )},\text{ \ \ }\forall t\in R,
\end{equation*}%
by the last differential inequality, we obtain%
\begin{equation*}
\frac{d}{dt}\Lambda _{k,s}(t)+\widehat{c}_{1}\Lambda _{k,s}(t)\leq \widehat{c%
}_{2}\Lambda _{k,s}(t)\left\Vert \nabla v_{kt}(t,s)\right\Vert
_{L_{2}(\Omega )}^{2}+
\end{equation*}
\begin{equation*}
+\widehat{c}_{4}k^{\frac{256}{65}}+\widehat{c}_{4}k^{8}\left\Vert \nabla
w_{t}(t)\right\Vert _{L_{2}(\Omega )}^{2},\text{ \ }\forall t\geq s.
\end{equation*}%
Multiplying both sides of the above inequality by $e^{\underset{s}{\overset{t%
}{\int }}\left[ \widehat{c}_{1}-\widehat{c}_{2}\left\Vert \nabla v_{kt}(\tau
,s)\right\Vert _{L_{2}(\Omega )}^{2}\right] d\tau }$ , integrating over $%
\left[ s,T\right] $, multiplying both sides of the obtained inequality by%
\newline
$e^{-\underset{s}{\overset{T}{\int }}\left[ \widehat{c}_{1}-\widehat{c}%
_{2}\left\Vert \nabla v_{kt}(t,s)\right\Vert _{L_{2}(\Omega )}^{2}\right]
dt} $ and taking into account (5.7), we find
\begin{equation*}
\Lambda _{k,s}(T)\leq \widehat{c}_{5}k^{\frac{32}{5}}\left\vert \left\langle
w_{t}(s),w^{(m)}(s)\right\rangle \right\vert +\widehat{c}_{5}k^{\frac{256}{65%
}}+
\end{equation*}%
\begin{equation}
+\widehat{c}_{5}k^{8}\underset{s}{\overset{T}{\int }}\left\Vert \nabla
w_{t}(t)\right\Vert _{L_{2}(\Omega )}^{2}dt,\text{ \ }\forall T\geq s,
\tag{5.16}
\end{equation}%
for large enough $k$ and small enough $\mu $. On the other hand, since $%
\mathcal{A}_{\mathcal{H}}$ is compact subset of $\mathcal{H}$ and problem
(1.1)-(1.3) admits a strict Lyapunov function, we have%
\begin{equation}
w_{t}(t)\rightarrow 0\text{ strongly in }L_{2}(\Omega )\text{ as }%
t\rightarrow -\infty  \tag{5.17}
\end{equation}%
Thus, by (5.6) and (5.17), for any $k\in
\mathbb{N}
$ there exists $T_{k}=T_{k}(\gamma )<0$ such that%
\begin{equation*}
\widehat{c}_{5}k^{\frac{32}{5}}\left\vert \left\langle
w_{t}(T),w^{(m)}(T)\right\rangle \right\vert +\widehat{c}_{5}k^{8}\underset{%
-\infty }{\overset{T}{\int }}\left\Vert \nabla w_{t}(t)\right\Vert
_{L_{2}(\Omega )}^{2}dt\leq 1,\text{ \ }\forall T\leq T_{k},
\end{equation*}%
which together with (5.16) yields (5.3).
\end{proof}

\begin{lemma}
\textit{Assume that conditions (2.1)-(2.3) are satisfied}. \textit{Then}
there exists $k_{0}\in
\mathbb{N}
$ such that \textit{\ }%
\begin{equation}
\underset{s\rightarrow -\infty }{\lim }\left( \left\Vert
u_{k_{0}t}(t,s)\right\Vert _{L_{2}(\Omega )}+\left\Vert
u_{k_{0}}(t,s)\right\Vert _{H^{1}(\Omega )}\right) =0,\text{ \ }\forall
t\leq T_{k_{0}}  \tag{5.18}
\end{equation}
\end{lemma}

\begin{proof}
Multiplying both sides of (5.2)$_{1}$ by $u_{kt}+\mu u_{k}$ ($\mu \in (0,1)$%
) and integrating over $\Omega $, we obtain%
\begin{equation*}
\frac{d}{dt}\left( E(u_{k}(t,s))+\frac{\mu }{2}\left\Vert \nabla
u_{k}(t,s)\right\Vert _{L_{2}(\Omega )}^{2}+\mu \left\langle
u_{kt}(t,s),u_{k}(t,s)\right\rangle \right) +
\end{equation*}%
\begin{equation*}
+\left\Vert \nabla u_{kt}(t,s)\right\Vert _{L_{2}(\Omega )}^{2}+\mu
\left\Vert \nabla u_{k}(t,s)\right\Vert _{L_{2}(\Omega )}^{2}-\mu \left\Vert
u_{kt}(t,s)\right\Vert _{L_{2}(\Omega )}^{2}\leq
\end{equation*}%
\begin{equation*}
\leq \left\Vert \sigma (w(t))-\sigma _{k}(w(t))\right\Vert _{L_{\frac{3}{2}%
}(\Omega )}\left\Vert v_{kt}(t,s)\right\Vert _{L_{6}(\Omega )}\left\Vert
u_{kt}(t,s)\right\Vert _{L_{6}(\Omega )}+
\end{equation*}%
\begin{equation*}
+\mu \left\Vert \sigma (w(t))\right\Vert _{L_{\frac{3}{2}}(\Omega
)}\left\Vert u_{kt}(t,s)\right\Vert _{L_{6}(\Omega )}\left\Vert
u_{k}(t,s)\right\Vert _{L_{6}(\Omega )}+
\end{equation*}%
\begin{equation*}
+\mu \left\Vert \sigma (w(t))-\sigma _{k}(w(t))\right\Vert _{L_{\frac{3}{2}%
}(\Omega )}\left\Vert v_{kt}(t,s)\right\Vert _{L_{6}(\Omega )}\left\Vert
u_{k}(t,s)\right\Vert _{L_{6}(\Omega )}+
\end{equation*}%
\begin{equation*}
+\left\Vert f(w(t))-f_{k}(w(t))\right\Vert _{L_{\frac{6}{5}}(\Omega
)}\left\Vert u_{kt}(t,s)\right\Vert _{L_{6}(\Omega )}+
\end{equation*}%
\begin{equation}
+\mu \left\Vert f(w(t))-f_{k}(w(t))\right\Vert _{L_{\frac{6}{5}}(\Omega
)}\left\Vert u_{k}(t,s)\right\Vert _{L_{6}(\Omega )},\text{ \ }\forall t\geq
s.  \tag{5.19}
\end{equation}%
Taking into account (2.4) in (5.19) and choosing $\mu $ small enough, we find%
\begin{equation*}
\frac{d}{dt}\left( E(u_{k}(t,s))+\frac{\mu }{2}\left\Vert \nabla
u_{k}(t,s)\right\Vert _{L_{2}(\Omega )}^{2}+\mu \left\langle
u_{kt}(t,s),u_{k}(t,s)\right\rangle \right) +
\end{equation*}%
\begin{equation*}
+c_{1}\left( E(u_{k}(t,s))+\frac{\mu }{2}\left\Vert \nabla
u_{k}(t,s)\right\Vert _{L_{2}(\Omega )}^{2}+\mu \left\langle
u_{kt}(t,s),u_{k}(t,s)\right\rangle \right) \leq
\end{equation*}%
\begin{equation*}
\leq c_{2}\left\Vert \sigma (w(t))-\sigma _{k}(w(t))\right\Vert _{L_{\frac{3%
}{2}}(\Omega )}^{2}\left\Vert v_{kt}(t,s)\right\Vert _{L_{6}(\Omega )}^{2}+
\end{equation*}%
\begin{equation}
+c_{2}\left\Vert f(w(t))-f_{k}(w(t))\right\Vert _{L_{\frac{6}{5}}(\Omega
)}^{2},\text{ \ }s\leq t\leq T_{k},  \tag{5.20}
\end{equation}%
where $c_{1}$ and $c_{2}$ are positive constants. Now let us estimate the
terms on the right side of (5.20). Since $H^{\frac{3}{2}+\varepsilon
}(\Omega )\subset C(\overline{\Omega })$ and%
\begin{equation*}
\left\Vert \varphi \right\Vert _{H^{\frac{3}{2}+\varepsilon }(\Omega )}\leq
c_{3}(\varepsilon )\left\Vert \varphi \right\Vert _{H^{1}(\Omega )}^{\frac{1%
}{2}-\varepsilon }\left\Vert \varphi \right\Vert _{H^{2}(\Omega )}^{\frac{1}{%
2}+\varepsilon },\text{ }\forall \varphi \in H^{2}(\Omega ),\text{ }\forall
\varepsilon \in (0,\frac{1}{2}],
\end{equation*}%
from (5.3) and (5.4) it follows that%
\begin{equation*}
\left\Vert v_{k}(t,s)\right\Vert _{C(\overline{\Omega })}\leq \frac{1}{2}k,%
\text{ \ }s\leq t\leq T_{k},
\end{equation*}%
for large enough $k$. The last inequality together with (2.1)-(2.4) yields
that%
\begin{equation*}
\left\Vert \sigma (w(t))-\sigma _{k}(w(t))\right\Vert _{L_{\frac{3}{2}%
}(\Omega )}^{\frac{3}{2}}\leq c_{4}\underset{\left\{ x:x\in {\small \Omega ,}%
\left\vert w(t,x)\right\vert >k\right\} }{\int }\left\vert w(t,x)\right\vert
^{6}dx\leq
\end{equation*}%
\begin{equation*}
\leq c_{5}\left( \underset{\left\{ x:x\in {\small \Omega ,}\left\vert
w(t,x)\right\vert >k\right\} }{\int }\left\vert w(t,x)\right\vert
^{6}dx\right) ^{\frac{1}{4}}\leq
\end{equation*}%
\begin{equation*}
\leq c_{5}\left( \underset{\left\{ x:x\in {\small \Omega ,}\left\vert
u_{k}(t,s,x)\right\vert >\left\vert v_{k}(t,s,x)\right\vert \right\} }{\int }%
\left\vert w(t,x)\right\vert ^{6}dx\right) ^{\frac{1}{4}}\leq
\end{equation*}%
\begin{equation*}
\leq c_{6}\left( \underset{\left\{ x:x\in {\small \Omega ,}\left\vert
u_{k}(t,s,x)\right\vert >\left\vert v_{k}(t,s,x)\right\vert \right\} }{\int }%
\left\vert u_{k}(t,s,x)\right\vert ^{6}dx\right) ^{\frac{1}{4}}\leq
\end{equation*}%
\begin{equation}
\leq c_{6}\left\Vert \nabla u_{k}(t,s)\right\Vert _{L_{2}(\Omega )}^{\frac{3%
}{2}},\text{ \ }s\leq t\leq T_{k},  \tag{5.21}
\end{equation}%
and
\begin{equation*}
\left\Vert f(w(t))-f_{k}(w(t))\right\Vert _{L_{\frac{6}{5}}(\Omega )}^{\frac{%
6}{5}}\leq c_{7}\underset{\left\{ x:x\in {\small \Omega ,}\left\vert
w(t,x)\right\vert >k\right\} }{\int }\left\vert w(t,x)\right\vert ^{6}dx\leq
\end{equation*}%
\begin{equation*}
\leq c_{8}\left( \underset{\left\{ x:x\in {\small \Omega ,}\left\vert
w(t,x)\right\vert >k\right\} }{\int }\left\vert w(t,x)\right\vert
^{6}dx\right) ^{\frac{4}{5}}\times
\end{equation*}%
\begin{equation*}
\times \left( \underset{\left\{ x:x\in {\small \Omega ,}\left\vert
u_{k}(t,s,x)\right\vert >\left\vert v_{k}(t,s,x)\right\vert \right\} }{\int }%
\left\vert w(t,x)\right\vert ^{6}dx\right) ^{\frac{1}{5}}\leq
\end{equation*}%
\begin{equation*}
\leq c_{9}\left( \underset{\left\{ x:x\in {\small \Omega ,}\left\vert
w(t,x)\right\vert >k\right\} }{\int }\left\vert w(t,x)\right\vert
^{6}dx\right) ^{\frac{4}{5}}\times
\end{equation*}%
\begin{equation*}
\times \left( \underset{\left\{ x:x\in {\small \Omega ,}\left\vert
u_{k}(t,s,x)\right\vert >\left\vert v_{k}(t,s,x)\right\vert \right\} }{\int }%
\left\vert u_{k}(t,s,x)\right\vert ^{6}dx\right) ^{\frac{1}{5}}\leq
\end{equation*}%
\begin{equation}
\leq c_{10}\left( \underset{\left\{ x:x\in {\small \Omega ,}\left\vert
w(t,x)\right\vert >k\right\} }{\int }\left\vert w(t,x)\right\vert
^{6}dx\right) ^{\frac{4}{5}}\left\Vert \nabla u_{k}(t,s)\right\Vert
_{L_{2}(\Omega )}^{\frac{6}{5}},\text{ \ \ }s\leq t\leq T_{k},  \tag{5.22}
\end{equation}%
for large enough $k$. On the other hand, since $\mathcal{A}_{\mathcal{H}}$
is compact subset of $\mathcal{H}$ and $(w(t),w_{t}(t))\in \mathcal{A}_{%
\mathcal{H}}$, we have%
\begin{equation}
\underset{t\in R}{\sup }\underset{\left\{ x:x\in {\small \Omega ,}\left\vert
w(t,x)\right\vert >k\right\} }{\int }\left\vert w(t,x)\right\vert
^{6}dx\rightarrow 0\text{ as }k\rightarrow \infty  \tag{5.23}
\end{equation}%
Thus choosing $\mu $ small enough, $k$ large enough and taking into account
(5.21)-(5.23) in (5.20), we obtain
\begin{equation*}
\frac{d}{dt}\widetilde{\Lambda }_{k,s}(t)+\widehat{c}_{1}\widetilde{\Lambda }%
_{k,s}(t)\leq \widehat{c}_{2}\left\Vert \nabla v_{kt}(t,s)\right\Vert
_{L_{2}(\Omega )}^{2}\widetilde{\Lambda }_{k,s}(t),\text{ }s\leq t\leq T_{k},
\end{equation*}%
where $\widehat{c}_{1}$ and $\widehat{c}_{2}$ are positive constants and $%
\widetilde{\Lambda }_{k,s}(t)=E(u_{k}(t,s))+\frac{\mu }{2}\left\Vert \nabla
u_{k}(t,s)\right\Vert _{L_{2}(\Omega )}^{2}\break+\mu \left\langle
u_{kt}(t,s),u_{k}(t,s)\right\rangle $. Now multiplying both sides of the
last inequality by\newline
$e^{\underset{s}{\overset{t}{\int }}\left[ \widehat{c}_{1}-\widehat{c}%
_{2}\left\Vert \nabla v_{kt}(\tau,s)\right\Vert _{L_{2}(\Omega )}^{2}\right]
d\tau }$ , integrating over $\left[ s,T_{k}\right] $ and multiplying both
sides of the obtained inequality by $e^{-\underset{s}{\overset{T_{k}}{\int }}%
\left[ \widehat{c}_{1}-\widehat{c}_{2}\left\Vert \nabla
v_{kt}(t,s)\right\Vert _{L_{2}(\Omega )}^{2}\right] dt}$, we find%
\begin{equation*}
\widetilde{\Lambda }_{k,s}(T)\leq \widetilde{\Lambda }_{k,s}(s)e^{-\underset{%
s}{\overset{T_{k}}{\int }}\left[ \widehat{c}_{1}-\widehat{c}_{2}\left\Vert
\nabla v_{kt}(t,s)\right\Vert _{L_{2}(\Omega )}^{2}\right] dt},\text{ }s\leq
t\leq T_{k},
\end{equation*}%
which together with (5.7) yields (5.18).
\end{proof}

By Lemma 5.1 and Lemma 5.2, we have $(w(T_{k_{0}}),w_{t}(T_{k_{0}}))\in $ $%
\mathcal{H}_{1}$ and
\begin{equation*}
\left\Vert w_{t}(T_{k_{0}})\right\Vert _{H^{1}(\Omega )}+\left\Vert
w(T_{k_{0}})\right\Vert _{H^{2}(\Omega )}\leq \widehat {r}_{0},
\end{equation*}%
where $\widehat {r}_{0}$ is independent of $(w_{0},w_{1})$. Now since $%
w(t,x) $ satisfies (1.1)-(1.3) on $(T_{k_{0}},\infty )\times \Omega $, with
initial data $(w(T_{k_{0}}),$ $w_{t}(T_{k_{0}}))$, applying Lemma 4.1 and
taking into account the last inequality, we find $(w_{0},w_{1})\in
(H^{2}(\Omega )\cap H_{0}^{1}(\Omega ))\times H_{0}^{1}(\Omega )$ and
\begin{equation*}
\left\Vert (w_{0},w_{1})\right\Vert _{H^{2}(\Omega )\times H^{1}(\Omega
)}\leq R_{0},
\end{equation*}%
where the positive constant $R_{0}$ is independent of $(w_{0},w_{1})$. So $%
\mathcal{A}_{\mathcal{H}}$ is a bounded subset of $(H^{2}(\Omega )\cap
H_{0}^{1}(\Omega ))\times H_{0}^{1}(\Omega )$ and that is why it coincides
with $\mathcal{A}_{\mathcal{H}_{1}}$.\newline

\medskip \medskip

\end{document}